\documentclass[a4paper,11pt]{article}
\usepackage{amssymb}
\usepackage{amsmath}
\usepackage{amsfonts}
\usepackage{amsthm, amsmath}
\usepackage{latexsym}
\usepackage[usenames]{color} 
\usepackage{bm}
\usepackage[normalem]{ulem}
\usepackage{cancel}
\usepackage{overpic}
\usepackage{graphicx}
\usepackage{natbib}
\usepackage{mathrsfs}
\usepackage{caption}
\usepackage{enumerate}
\usepackage{exscale}
\usepackage[latin9]{inputenc}

\def\Om{\Omega}
 \def\ua{\uparrow}
 \def\da{\downarrow}

 \def\wh{\widehat}
 \def\wt{\widetilde}

\def\bbar{\overline}
\def\bR{\mathbb{R}}

\def\Om{\Omega}\def\om{\omega}

\def\ignore#1{}

\def\bE{\mathbb E}
\def\cF{{\mathscr F}}

\def\cL{\mathscr L}

\def\cR{{\mathscr R}}

\def\cX{{\mathscr X}}

\def\cR{\mathscr R}

\def\bN{\mathbb N}

\def\bP{{\mathbb P}}
\def\bE{{\mathbb E}}

\def\<{\langle}
\def\>{\rangle}

\def\ua{\uparrow}
\def\da{\downarrow}

\def\var{\text{var\,}}

\newtheorem{theorem}{Theorem}[section]
\newtheorem{proposition}[theorem]{Proposition}
\newtheorem{lemma}[theorem]{Lemma}
\newtheorem{corollary}[theorem]{Corollary}

\theoremstyle{definition}
\newtheorem{definition}[theorem]{Definition}

\newtheorem{remark}[theorem]{Remark}

\def\wt{\widetilde}
\def\bbar{\overline}

\normalsize

\begin{document}

\title{{\LARGE\bf  A state-constrained differential game \\ arising in optimal portfolio liquidation}}
\author{\normalsize Alexander
Schied\thanks{University of Mannheim, Department of Mathematics, A5, 6, 68131 Mannheim, Germany, {\tt schied@uni-mannheim.de}}{\setcounter{footnote}{6}}\and  Tao Zhang\thanks{University of Mannheim, Department of Mathematics, A5, 6, 68131 Mannheim, Germany\hfill\break
The authors acknowledge support by Deutsche Forschungsgemeinschaft through Research Grant SCHI 500/3-1}}

\date{\normalsize First version: December 27, 2013\\
This version: July 7, 2015}

\maketitle

\begin{abstract}
We consider $n$ risk-averse agents who compete for liquidity in an Almgren--Chriss market impact model. Mathematically, this situation can be described by a Nash equilibrium for a certain linear-quadratic differential game with state constraints. The state constraints enter the problem as terminal boundary conditions for finite and infinite time horizons. We prove existence and uniqueness of Nash equilibria and give  closed-form solutions in some special cases. We also analyze qualitative properties of the equilibrium strategies and provide corresponding financial interpretations.\end{abstract}
 
 \medskip



\section{Introduction}

In this paper, we analyze a state-constrained differential game  that arises for risk-averse  agents  aiming to liquidate a given asset position by a given time $T>0$. Agents   face  both price impact and volatility risk.  For each  agent, there is hence a tradeoff between slow trading so as to reduce transaction costs from price impact and fast liquidation in view of volatility risk. Beginning with \citet{BertsimasLo} and \citet{AlmgrenChriss2}, a large numbers of papers have studied 
the corresponding single-agent optimization problems in various settings; see \citet{Lehalle} and \citet{GatheralSchiedSurvey} for recent overviews and more complete lists of references. The problem becomes even more interesting when considering not just one, but $n$ agents who are aware of each others initial positions, a situation that is not unlikely to occur in reality; see \citet{Carlinetal} and \citet{SchoenebornSchied}. Together with \citet{BrunnermeierPedersen}, these two papers were  among the first to consider a   game theoretic approach, but only study open-loop Nash equilibria for risk-neutral agents using deterministic strategies. \citet{Moallemietal} extend the analysis to a model with asymmetric information.  \citet{CarmonaYang} use numerical simulations to study a system of coupled HJB equations arising from a closed-loop Nash equilibrium for two utility-maximizing agents. \cite{Lachapelleetal} apply mean-field games to model  the price formation process in the presence of high-frequency traders. A two-player Nash equilibrium in a market impact model with exponentially decaying transient price impact is analyzed in \citet{SchiedZhangHotPotato}.

Here, we consider agents maximizing a mean-variance functional in a continuous-time \citet{AlmgrenChriss2} framework, which is  a   common  setup for portfolio liquidation. It leads to a linear-quadratic differential game, which has the interesting additional feature of a terminal state constraint arising from the  liquidation constraint on the portfolio. This state constraint leads to two-point boundary problems in place of the usual initial value problems connected with unconstrained differential games. Aside from the financial interpretation of our results, this paper thus also provides a natural case study for  a class  state-constrained differential games. 

Our main results establish existence and uniqueness  for the corresponding Nash equilibria with both finite and infinite time horizon. In  several cases, we can also give closed-form solutions for the equilibrium strategies. These formulas enable us to discuss some qualitative properties of the Nash equilibrium. Some of these properties are surprising, as they show that certain monotonicity properties that are discussed in the finance literature may break down under certain market conditions. See \citet{Maug} for discussions and for an empirical analysis of a large data set of portfolio liquidations of large investors.

The paper is organized as follows. In Section~\ref{Finite time background section},
 we recall some background material on portfolio liquidation in the Almgren--Chriss framework. Existence, uniqueness, and representation results for Nash equilibria with finite time horizon are stated in Section~\ref{Finite time Nash section}. Section~\ref{Finite time qualitative section} contains a discussion of the qualitative properties of the corresponding two-player Nash equilibrium.  Nash equilibria with infinite time horizon are discussed in Section~\ref{infinite-horizon section}. All proofs can be found in Section~\ref{Proofs Section}.

\section{Nash equilibrium with finite time horizon
}\label{Finite time section}
\subsection{Background}\label{Finite time background section}
We consider  a standard continuous-time  \citet{AlmgrenChriss2} framework for investors who are active over a fixed time period $[0,T]$. 
An investor may hold an initial position of $x$ shares and is required to close this position by time $T$.  
The information flow available to an investor is modeled by a filtration $(\cF_t)_{t\ge0}$ on a given probability space $(\Om,\cF,\bP)$. The trading strategy employed by the investor is denoted by $X=(X(t))_{t\in[0,T]}$. It needs to satisfy the following conditions of admissibility:
\\
$\bullet$ $X$ satisfies  the liquidation constraint $X(T)=0$;\\
$\bullet$  $X$ is adapted to the filtration  $(\cF_t)_{t\ge0}$;\\
$\bullet$  $X$ is absolutely continuous in the sense that there exists a progressively measurable process $(\dot X(t))_{t\in[0,T]}$ such that for all $\om\in\Om$, $\int_0^T(\dot X(t,\om))^2\,dt<\infty$ and 
$$X(t,\om)=X(0,\omega)+\int_0^t\dot X(s,\om)\,ds,\qquad t\in[0,T];
$$
$\bullet$  there exists a constant $c\ge0$ such that $|X(t,\om)|\le c$ for all $t$ and $\om$. \\
The class of all strategies that are admissible in this sense and satisfy $X(0)=x$ for given $x\in\bR$ will be denoted by $\cX(x,T)$. Let us also introduce the subclass $\cX_{\text{det}}(x,T)$ of all  strategies in $\cX(x,T)$ that are {deterministic} in the sense that they do not depend on $\om$.
The \lq unaffected price process\rq\ $S^0$ will describe the fluctuations of asset prices perceived by an investor who has no inside information on large trades carried out by other market participants during the time interval  $[0,T]$. In the Almgren--Chriss model, it is usually assumed that $S^0$  follows a Bachelier model. Here we are sometimes also going to allow for an extra drift to describe current price trends. Thus,
\begin{align*}
S^0(t)=S_0+\sigma W(t)+\int_0^tb(s)\,ds,
\end{align*}
where $S_0$ is a constant, $W$ is a standard Brownian motion, $\sigma\ge0$, and $b$ is deterministic and continuous. 

When an investor is using a strategy $X\in\cX(x,T)$, the strategy $X$ will influence the prices at which assets are traded. In the linear  Almgren--Chriss framework,  the resulting  price is assumed to be
\begin{align}\label{affected price process one player}
S^X(t):=S^0(t)+\gamma(X(t)-X(0))+\lambda\dot X(t),\qquad t\in[0,T],
\end{align} 
where the  constants $\gamma\ge0$ and $\lambda>0$ describe the  permanent and temporary price impact components.
At each time $t\in[0,T]$, the infinitesimal amount of $-\dot X(t)\,dt$ shares are sold at price $S^X(t)$. The total revenues generated by the strategy  $X\in\cX(x,T)$ are therefore given by
\begin{equation*}
\begin{split}
\cR(X)&:=-\int_0^T\dot X(t)S^X(t)\,dt.
\end{split}
\end{equation*}
 The optimal trade execution problem consists in maximizing a cost-risk functional of the revenues over all admissible strategies in $\cX(x,T)$. One possibility is the maximization of expected revenues,
\begin{align}\label{expected revenues eq}
\text{maximize}\quad \bE[\,\cR(X)\,],
\end{align}
as considered in many papers on optimal execution and, with the notable exception of \citet{CarmonaYang}, all other papers dealing with corresponding Nash equilibria. \citet{BertsimasLo} were  among the first to propose the problem \eqref{expected revenues eq}. In practice, it is common to account for the volatility risk arising from late execution by maximizing a mean-variance criterion:
\begin{align}\label{mean-variance criterion}
\text{maximize}\quad  \bE[\,\cR(X)\,]
-\frac\alpha2\,\var(\cR(X));
\end{align}
here $\alpha$ is a nonnegative risk-aversion parameter. 
When dealing  with the problem~\eqref{mean-variance criterion}, admissible strategies are usually restricted  to the class $\cX_{\text{det}}(x,T)$ of deterministic strategies; see \citet{AlmgrenChriss2} and \citet{Almgren}. Except for the results in \citet{LorenzAlmgren}, little is known when general adapted strategies are used in~\eqref{mean-variance criterion}; best of the authors' knowledge, not even the existence of maximizers has been established to date.
The main reason for this  is the lack of time consistency of the variance functional, which does not fit well into a dynamic optimization context.
 On the other hand,
 \cite{SchiedSchoenebornTehranchi} show that the maximization of~\eqref{mean-variance criterion}
 over deterministic strategies $X\in\cX_{\text{det}}(x,T)$ is equivalent to the maximization of the expected utility of revenues,
 \begin{align}\label{expected utility eq}
 \text{maximize}\quad \bE[\,u_\alpha(\cR(X))\,],
\end{align}
over all strategies in $\cX(x,T)$, when 
\begin{align}\label{ualpha}
u_\alpha(x):=\begin{cases}\frac1\alpha(1-e^{-\alpha x})&\text{if $\alpha>0$,}\\
x&\text{if $\alpha=0$,}
\end{cases}
\end{align}
is a CARA utility function with absolute risk aversion $\alpha\ge 0$. See \citet{Lehalle} and \citet{GatheralSchiedSurvey} for recent overviews on portfolio liquidation and related  market microstructure issues.

\subsection{Nash equilibrium}\label{Finite time Nash section}

Now suppose that $n$ investors are active in the market, using the respective strategies $X_1,\dots,X_n$. As in~\eqref{affected price process one player}, each strategy $X_i$ will impact the price process $S^0$, thus leading to the following price with aggregated price impact:
\begin{align}\label{affected price process n players eq}
S^{X_1,\dots,X_n}(t):=S^0(t)+\gamma\sum_{j=1}^n(X_j(t)-X_j(0))+\lambda\sum_{j=1}^n\dot X_j(t),\qquad t\in[0,T].
\end{align}
Let us denote by $\bm X_{-i}:=\{X_1,\dots, X_{i-1},X_{i+1},\dots, X_n\}$ the collection   of the strategies of all competitors of player $i$. Then, player $i$ will obtain the following revenues,
$$
\cR(X_i|\bm X_{-i})=-\int_0^T\dot X_i(t)S^{X_1,\dots,X_n}(t)\,dt,
$$
and seek to maximize one of the objective functionals~\eqref{expected revenues eq},~\eqref{mean-variance criterion}, or~\eqref{expected utility eq}. A natural question is whether there exists a Nash equilibrium in which all players maximizes their objective functionals given the strategies of their competitors. For the maximization of the expected revenues and vanishing drift,  this problem is solved in \citet{Carlinetal} within the class of deterministic strategies. It was later extended in \citet{SchoenebornSchied} to the case in which players have different time horizons and in \citet{Moallemietal} to a situation with asymmetric information. A system of coupled HJB equations arising from a closed-loop Nash equilibrium for two utility-maximizing agents is studied through numerical simulations by \citet{CarmonaYang}. Here, we will  conduct a mathematical analysis of $n$-player open-loop Nash equilibria for  mean-variance optimization~\eqref{mean-variance criterion} and CARA utility maximization~\eqref{expected utility eq}. 

\medskip

\begin{definition}\label{Nash eq def}Suppose that $n\in\bN$, $x_1,\dots,x_n\in\bR$ are initial asset positions,  and  $\alpha_1,\dots,\alpha_n$ are nonnegative  coefficients of risk aversion. \\
(a) A \emph{Nash equilibrium for mean-variance optimization} consists of a collection $X_1^*,\dots, X_n^*$ of deterministic strategies such that, for each $i$  and  ${\bm X}^*_{-i}=\{X^*_1,\dots, X^*_{i-1},X^*_{i+1},\dots, X^*_n\}$, the strategy $X^*_i\in \cX_{\text{det}}(x_i,T)$ maximizes the mean-variance functional 
$$ \bE[\,\cR(X|{\bm X}^*_{-i})\,]
-\frac{\alpha_i}2\,\var(\cR(X|{\bm X}^*_{-i}))
$$
over all  $X\in \cX_{\text{det}}(x_i,T)$.\\
(b) A \emph{Nash equilibrium for CARA utility maximization} consists of a collection $X_1^*,\dots, X_n^*$ of admissible strategies such that, for each $i$,   the strategy $X^*_i\in \cX(x_i,T)$ maximizes the expected utility 
$$ \bE[\,u_{\alpha_i}(\cR(X|{\bm X}^*_{-i}))\,]
$$
over all  $X\in \cX(x_i,T)$.
\end{definition}


Note that the equilibrium strategies $X^*_i$ for CARA utility maximization are allowed to be adapted, whereas, for reasons explained above, only deterministic strategies are admitted in mean-variance optimization. 
We start by formulating a general  existence and uniqueness result for the Nash equilibrium for mean-variance optimization.


\begin{theorem}\label{Nash Eq exist thm}For given $n\in\bN$, $\alpha_1,\dots,\alpha_n\ge0$, and $x_1,\dots,x_n$, there exists a unique Nash equilibrium $X_1^*,\dots, X_n^*$ for mean-variance optimization. It is given as the unique solution of the following second-order system of differential equations 
\begin{equation}\label{ode system 1 thm}
 {\alpha_i}\sigma^2X_i(t)
-2\lambda\ddot X_i(t)=b(t)+\gamma\sum_{j\neq i}\dot X_j(t)+\lambda \sum_{j\neq i}\ddot X_j(t)\end{equation}
with two-point boundary conditions 
\begin{align}\label{ode system 1 boundary conditions thm}
\text{$X_i(0)=x_i$ and $X_i(T)=0$}
\end{align}
 for $i=1,2,\dots,n$. \end{theorem}


It will become clear from~\eqref{mean-variance = action functional eq} and~\eqref{Lagrangian} below that, from a mathematical point of view, the Nash equilibrium constructed above corresponds to an open-loop linear-quadratic differential game with state constraints. The state constraints are provided by the liquidation constraints $X_i(T)=0$, $i=1,\dots, n$. They    are responsible for the fact that 
 we cannot apply standard results on the existence and uniqueness of open-loop linear-quadratic differential games, and significantly complicate the proof for the existence of Nash equilibria, especially in the case of an infinite time horizon as studied in Section~\ref{infinite-horizon section}. It may also be of interest  that the proof of the existence of solutions to~\eqref{ode system 1 thm},~\eqref{ode system 1 boundary conditions thm} rests on  the uniqueness of Nash equilibria, which will be established in Lemma~\ref{Equilibrium uniqueness lemma} below.

Our next result states that the unique Nash equilibrium for mean-variance optimization is also a Nash equilibrium for CARA utility maximization. It is an open question, however, whether there may be more than one Nash equilibrium for CARA utility maximization.


\begin{corollary}\label{CARA Nash Eq exist cor}For given $n\in\bN$, $\alpha_1,\dots,\alpha_n\ge0$, and $x_1,\dots,x_n$, the Nash equilibrium for mean-variance optimization constructed in Theorem~\ref{Nash Eq exist thm} is also a Nash equilibrium for CARA utility maximization.
\end{corollary}


Let $(\wt\cF_t)_{t\ge0}$ be any sub-filtration of $(\cF_t)_{t\ge0}$. It will follow from the proof of Corollary~\ref{CARA Nash Eq exist cor} that the Nash equilibrium for mean-variance optimization constructed in Theorem~\ref{Nash Eq exist thm} is also a Nash equilibrium for CARA utility maximization within the class of  all strategies that are adapted to $(\wt\cF_t)_{t\ge0}$. In particular, it is a Nash equilibrium  for CARA utility maximization within the class of  deterministic strategies.

Let us now have a closer look at 
the system~\eqref{ode system 1 thm}. It
 simplifies  when all agents have the same risk aversion.


 \begin{corollary}\label{drift cor}In the setting of Theorem~\ref{Nash Eq exist thm}, suppose that $\alpha_1=\cdots=\alpha_n=\alpha\ge0$. Then 
$$
 \Sigma(t):=\sum_{i=1}^nX^*_i(t)
$$
is the  unique solution  of the following one-dimensional two-point boundary value problem,
 \begin{align}\label{Sigma b bvp}
 \alpha\sigma^2\Sigma(t)-(n-1)\gamma\dot\Sigma(t)-(n+1)\lambda\ddot\Sigma(t)=nb(t),\qquad \Sigma(0)=\sum_{i=1}^nx_i,\ \Sigma(T)=0.
\end{align}
Given  $\Sigma$, each equilibrium strategy $X^*_i$ is  equal to  the unique solution of the following one-dimensional two-point boundary value problem,
\begin{align}\label{Xi b bvp}
\alpha\sigma^2X_i(t)+\gamma\dot X_i(t)-\lambda\ddot X_i(t)=b(t)+\gamma\dot\Sigma(t)+\lambda\ddot\Sigma(t),\qquad X_i(0)=x_i,\ X_i(T)=0.
\end{align}
\end{corollary}


It is possible to obtain closed-form solutions of~\eqref{Sigma b bvp} and~\eqref{Xi b bvp}, but the corresponding expressions are quite involved. The situation simplifies when  the drift $b$ vanishes identically. 


\begin{theorem}\label{alpha n=alpha b=0 thm}
In the setting of Corollary~\ref{drift cor}, assume that, in addition, $b=0$ and  $\alpha>0$. For
\begin{align}\label{wh rho-}
\wh\theta=\frac{\sqrt{\gamma^2+4\alpha\sigma^2\lambda}}{2\lambda}\qquad\text{and}\qquad
\wh\rho=\frac{\sqrt{(n-1)^2\gamma^2+4(n+1)\alpha\sigma^2\lambda}}{2(n+1)\lambda},
\end{align}
we define
\begin{align}\label{thetarho+-}
\theta_\pm=\frac{\gamma}{2\lambda}\pm\wh\theta\quad\text{and}\quad\rho_\pm=-\frac{(n-1)\gamma}{2(n+1)\lambda}\pm\wh\rho.
\end{align}
 Then, the $i^{\text{th}}$ equilibrium strategy $X^*_i$  is of the form
 \begin{align}\label{linear comb rep}
 X^*_i(t)=c_i(\theta_+)e^{\theta_+t}+c_i(\theta_-)e^{\theta_-t}+c(\rho_+)e^{\rho_+t}+c(\rho_-)e^{\rho_-t},
\end{align}
where, for $\bbar x_n:=\frac1n\sum_{j=1}^nx_j$,
\begin{align}\label{ci(rho pm)}
c_i(\theta_+)= \frac{\bbar x_n-x_i}{e^{2\wh \theta T}-1},\quad c_i(\theta_-)= \frac{-(\bbar x_n-x_i)}{1-e^{-2\wh \theta T}},\quad
c(\rho_+)=\frac{-\bbar x_n}{e^{2\wh\rho T}-1},\quad c(\rho_-)=\frac{\bbar x_n}{1-e^{-2\wh\rho T}}.\end{align}
Moreover,  the solution $ \Sigma(t)=\sum_{i=1}^nX^*_i(t)$ of the two-point boundary value problem~\eqref{Sigma b bvp} is given by
\begin{align}\label{Sigma general n eq}
\Sigma(t)=\frac{n\bbar x_n}{2\sinh(\wh\rho T)}\Big(e^{\wh\rho T}e^{\rho_-t}-e^{-\wh \rho T}e^{\rho_+t}\Big).
\end{align} 

\end{theorem}


The formulas in  Theorem~\ref{alpha n=alpha b=0 thm} can be further simplified in a two-player setting:


\begin{corollary}\label{n=2 Nash eq cor}In the setting of Theorem~\ref{alpha n=alpha b=0 thm}, assume in addition that $n=2$. Then
\begin{align}\label{X1*X2*}
X^*_1(t)=\frac12\big(\Sigma(t)+\Delta(t)\big)\qquad\text{and}\qquad X^*_2(t)=\frac12\big(\Sigma(t)-\Delta(t)\big),
\end{align}
where 
\begin{align}\label{Sigma for n=2}
\Sigma(t)&=(x_1+x_2)e^{-\frac{\gamma t}{6\lambda}}\frac{\sinh\Big(\frac{(T-t)\sqrt{\gamma^2+12\alpha\lambda\sigma^2}}{6\lambda}\Big)}{\sinh\Big(\frac{T\sqrt{\gamma^2+12\alpha\lambda\sigma^2}}{6\lambda}\Big)},\\
\Delta(t)&=(x_1-x_2)e^{\frac{\gamma t}{2\lambda}}\frac{\sinh\Big(\frac{(T-t)\sqrt{\gamma^2+4\alpha\lambda\sigma^2}}{2\lambda}\Big)}{\sinh\Big(\frac{T\sqrt{\gamma^2+4\alpha\lambda\sigma^2}}{2\lambda}\Big)}.\label{Delta for n=2}
\end{align}
\end{corollary}

The following mean-field limit is obtained in a straightforward manner by sending $n$ to infinity in Theorem~\ref{alpha n=alpha b=0 thm}.

\begin{corollary}\label{mean-field limit cor}In the setting of Theorem~\ref{alpha n=alpha b=0 thm}, suppose that $\lim_{n\ua\infty}\frac1n\sum_{j=1}^nx_j=\bbar x\in\bR$. Then, as $n\ua\infty$,  the equilibrium strategy of agent $i$ converges to  
\begin{align*}
 \frac{\bbar x-x_i}{e^{2\wh \theta T}-1}e^{\theta_+t}-  \frac{\bbar x-x_i}{1-e^{-2\wh \theta T}}e^{\theta_-t}+\frac{\bbar x}{1-e^{-\frac{\gamma T}\lambda}}e^{-\frac{\gamma t}\lambda}-\frac{\bbar x}{e^{\frac{\gamma T}\lambda}-1},
\end{align*}
where $\theta_+$, $\theta_-$, and $\wh\theta$ are as in~\eqref{thetarho+-} and~\eqref{wh rho-}.
\end{corollary}

Note that the mean-field limit in the preceding corollary need not correspond to an infinite-player equilibrium. For instance, if $x_i=1$ for all $i$, then the conditions of Corollary~\ref{mean-field limit cor} are satisfied, but the combined price impact of all players will be infinite so that the price process in the infinite-player limit does not exist. A discussion of infinite-player equilibria for market impact games will be left for future research.  

\subsection{Qualitative discussion of the two-player Nash equilibrium}\label{Finite time qualitative section}

Throughout this section, $(X^*_1,X^*_2)$ will denote the two-player Nash equilibrium constructed in  Corollary~\ref{n=2 Nash eq cor}. It is interesting to compare the strategies $X^*_i$  with the optimal strategy of a single agent without competitors, which, as observed by \citet{Almgren}, is given by
\begin{align*}
X_0^*(t)=x_0\frac{\sinh(\kappa(T-t))}{\sinh(\kappa T)},
\end{align*}
where $x_0$ is the initial asset position and $\kappa=\sqrt{\alpha\sigma^2/2\lambda}$. This formula can also be obtained  by taking $n=1$ in~\eqref{Sigma general n eq}. To study the behavior of the strategies $X^*_0,X^*_1,X^*_2$, we will need the following  elementary  fact, whose proof is left to the reader.
 \begin{align}\label{sinh fact}
\text{ For $0<\nu<1$ the function $\displaystyle x\longmapsto\frac{\sinh(\nu x)}{\sinh(x)}$ is strictly decreasing on $[0,\infty)$.}
\end{align}
 It follows immediately from this fact that $X^*_0(t)$ is  a strictly decreasing function of $\alpha \sigma^2$ if $x_0>0$ and $0<t<T$. Economically, this  means that the agent will liquidate the initial asset position faster when the perceived volatility risk increases, because $\var(\cR(X^*_0))$ is proportional to $\alpha\sigma^2$ according to~\eqref{mean-variance = action functional eq} and~\eqref{Lagrangian} below. So the first guess would be that  the equilibrium strategy $X_1^*$  should also be a decreasing function of $\alpha\sigma^2$ when $x_1>0$. This guess is  analyzed and tested empirically by  \citet{Maug} for a large data set of block executions by large insiders. In our equilibrium model, however, 
  all we get from applying~\eqref{sinh fact} to~\eqref{X1*X2*} is the following partial result. 


\begin{proposition}\label{alphasigma2 prop}If  $x_1\ge x_2\ge0$, then $X^*_1(t)$ is a strictly decreasing function of $\alpha\sigma^2$ for $0<t<T$.
\end{proposition}


As a matter of fact, the monotonicity in $\alpha\sigma^2$ may break down in the two-player Nash equilibrium if  the conditions $x_1\ge x_2$ and $x_2\ge0$ in Proposition~\ref{alphasigma2 prop}  are not both satisfied; see Figures~\ref{alphasigma2 fig1}
and~\ref{alphasigma2 fig2}. An intuitive explanation for this failure of monotonicity 
 is provided in  Figure~\ref{alphasigma explain fig}. 

Next, $X_0^*(t)$ is independent of $\gamma$, whereas both two-player equilibrium strategies are nontrivial functions of $\gamma$. The intuitive reason for this dependence is the fact that the permanent price impact created by the liquidation strategy of one agent is perceived as an additional price trend by the other agent.  

Moreover, $X_0^*(t)$ is an increasing function of $\lambda$ by~\eqref{sinh fact}.  The monotonicity in $\lambda$ has the clear economic intuition that increasing the transaction costs from temporary price impact reduces the benefits from an early liquidation and thus drives the optimal strategy toward the linear liquidation strategy, which is optimal in the risk-neutral case $\alpha=0$. The monotonicity of liquidation strategies as a function of $\lambda$ is   tested and analyzed empirically by \citet{Maug}. 
In our equilibrium model, 
the monotonic dependence of $X^*_1(t)$ on $\gamma$ and $\lambda$ can be obtained by applying~\eqref{sinh fact} to~\eqref{X1*X2*}, but only  if $x_1=x_2$. We thus get the following result. 


\begin{proposition}\label{gamma lambda prop}If $x_1=x_2\ge0$ then $X^*_1(t)=X_2^*(t)$ is a strictly decreasing function of $\gamma$ and a strictly increasing function of $\lambda$ for $0<t<T$.
\end{proposition}


 As shown in Figures~\ref{lambda fig} and~\ref{gamma fig}, the monotonic dependence on $\gamma$ or $\lambda$ may break down if the condition $x_1=x_2$ from Proposition~\ref{gamma lambda prop} is not satisfied. The intuitive explanation for these effects is similar to the one for the breakdown of monotonicity for $\alpha\sigma^2$. For instance, when $\lambda$ increases in a Nash equilibrium with $0<x_1\ll x_2$, both agents receive an incentive to reduce the curvature of their strategies, that is, to sell slower in the first part of the trading interval and to sell faster during the second part. Agent~2 will therefore create less price impact during the first part of $[0,T]$ and more price impact in the second part. In equilibrium, this change in price impact generated by one trader also  creates a second, competing incentive for the other trader, namely to increase trading speed during the first part of $[0,T]$ and to reduce it during the second part when the unfavorable price impact generated by the competitor is increased. When the position of agent~1 is smaller than the one of agent~2, this second incentive can dominate the first one quantitatively and hence trigger a decrease of $X^*_1(1)$, as observed in Figure~\ref{lambda fig} for  $\lambda<0.05$.

\bigskip

\hspace{-0.7cm}
\begin{minipage}[b]{8cm}
\centering
\includegraphics[width=5cm]{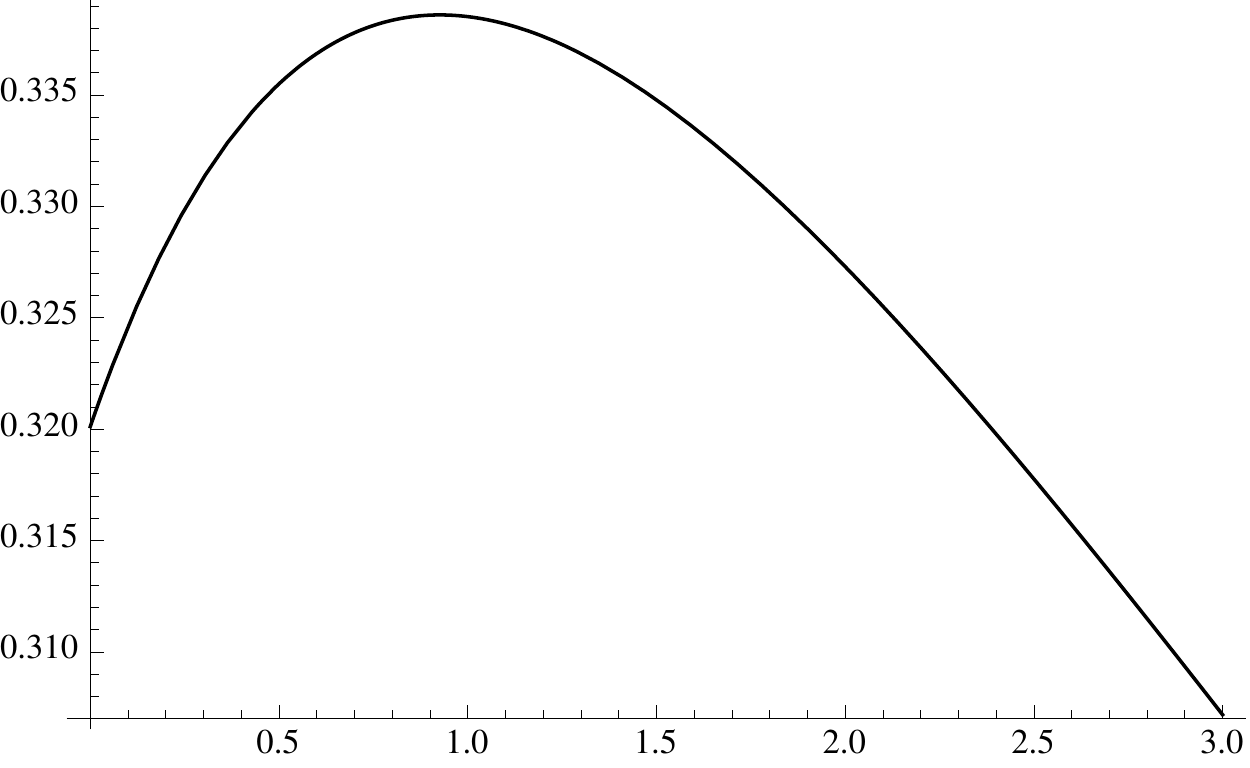}
\captionof{figure}{
 $X^*_1(1)$  as a function of $\alpha\sigma^2$ for $x_1=1.12 $, $x_2= 2.06$, $T=2$, and $\lambda=\gamma=1$.}
\label{alphasigma2 fig1}
\end{minipage}\qquad
\begin{minipage}[b]{8cm}
\centering
\includegraphics[width=5cm]{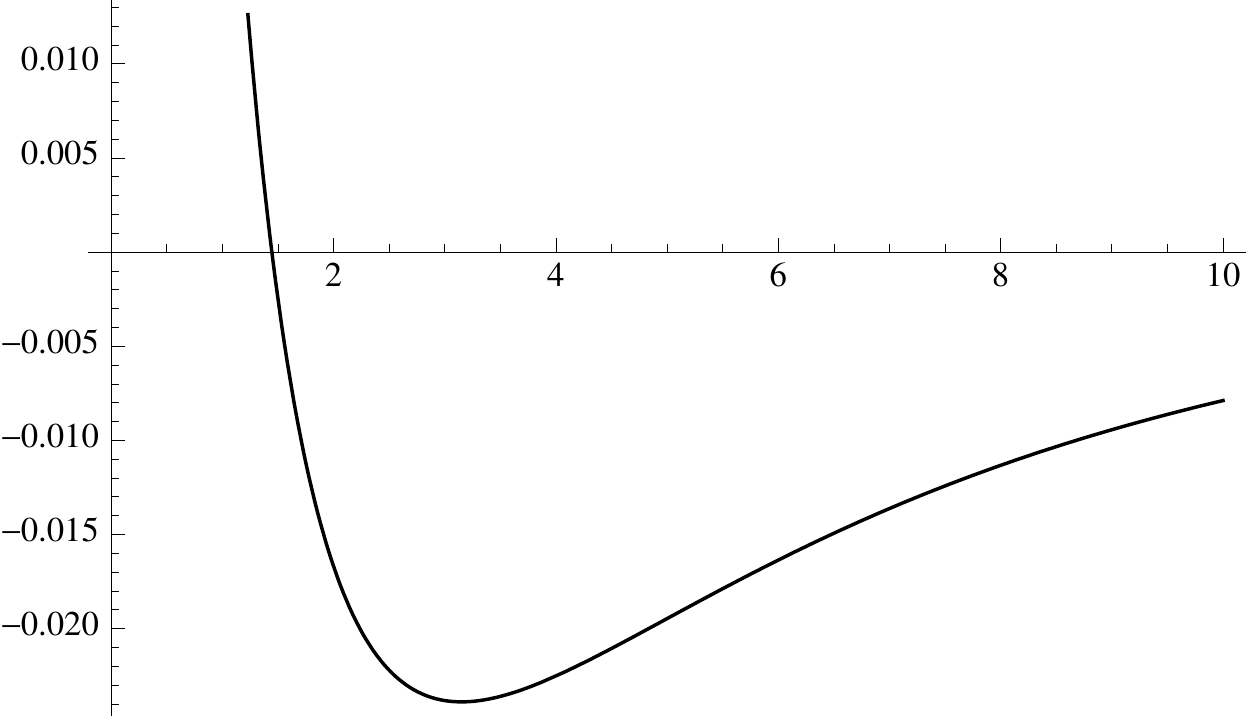}
\captionof{figure}{ $X^*_1(1)$  as a function of $\alpha\sigma^2$ for $x_1=0.7 $, $x_2= -1.9$, $T=2$,  $\lambda=0.2$, and $\gamma=0.1$.}
\label{alphasigma2 fig2}
\end{minipage}

\begin{figure}[h]
\centering\includegraphics[width=4.9cm]{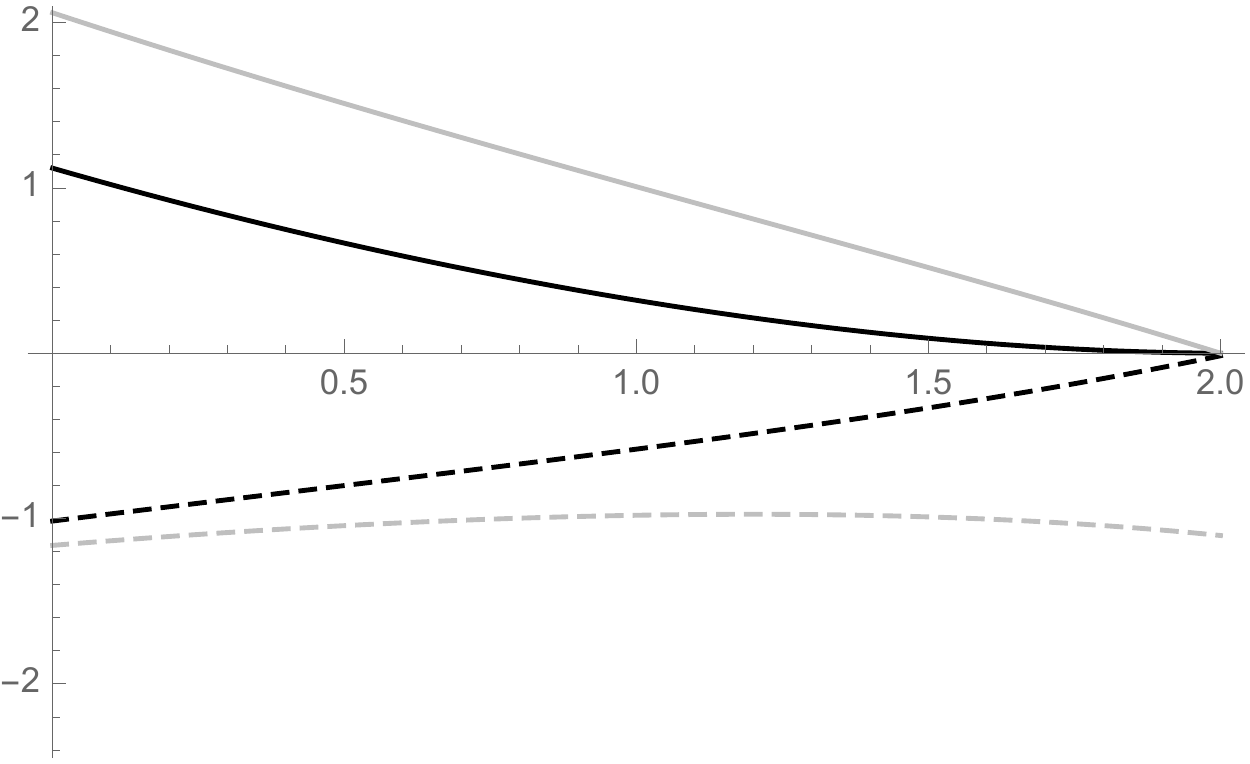}\qquad
\includegraphics[width=4.9cm]{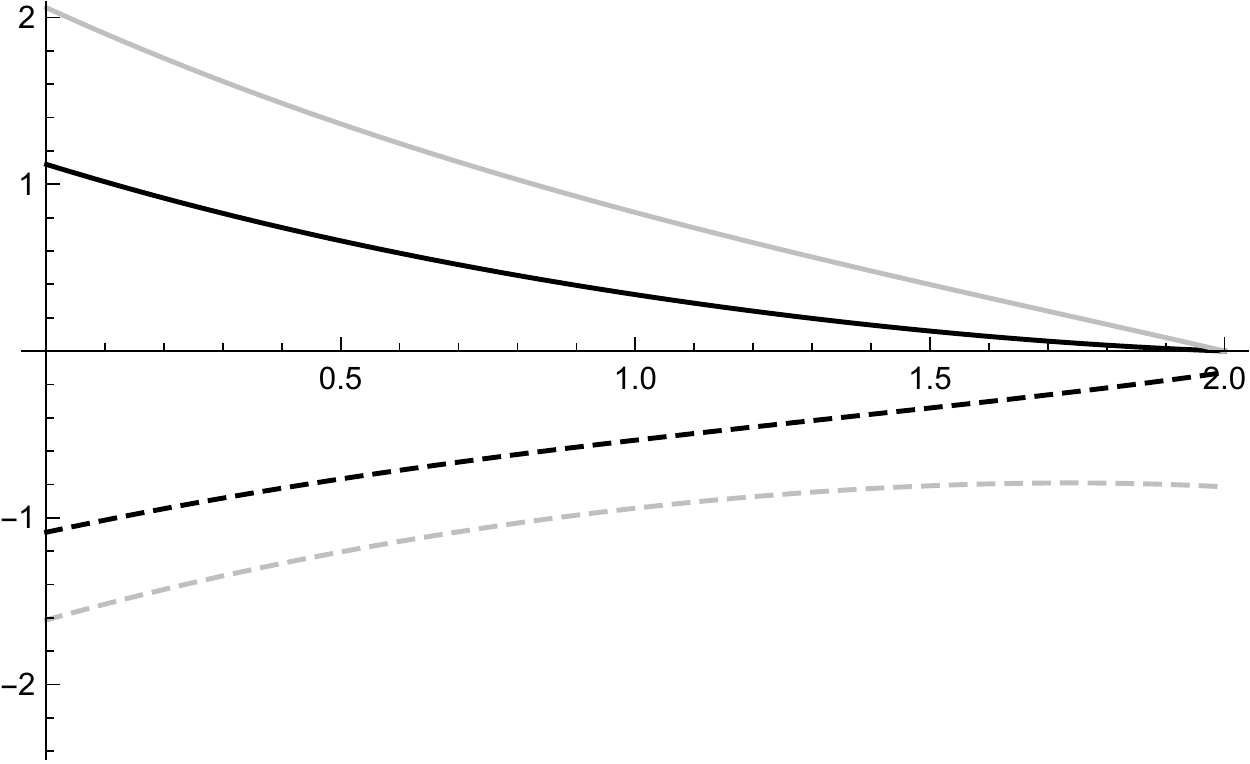}\qquad
\includegraphics[width=4.9cm]{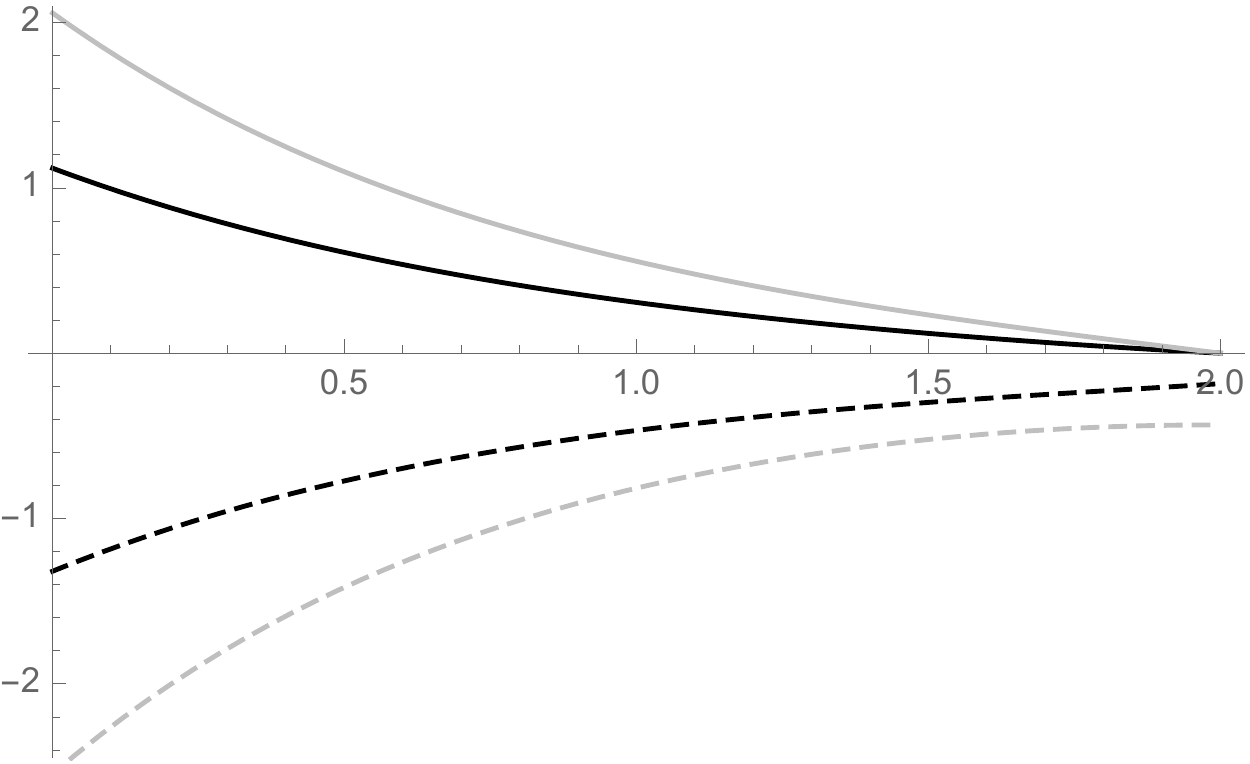}
\caption{Equilibrium strategies $X^*_i(t)$ (solid) and trading rates
 $\dot X^*_i(t)$ (dashed) for agents $i=1$ (black) and $i=2$ (grey)  as functions of $t\in[0,T]$ for  $\alpha\sigma^2=0$ (left), $\alpha\sigma^2=0.8$ (center),  and $\alpha\sigma^2=3$ (right); the remaining parameters are as in Figure~\ref{alphasigma2 fig1}.  When $\alpha\sigma^2$ is increased from 0  to 0.8, agent~2  receives a relatively high increase in volatility risk and therefore speeds up liquidation throughout the first half of $[0,T]$, while  slowing down in the second half. The volatility risk of agent~1 also increases, but it does so less than for agent~2 and leads only to a small initial increase of the liquidation speed $-\dot X_1^*(t)$. On the other hand, the increased price pressure from the temporary impact of agent~2 results in  unfavorable asset prices for agent~1 in  the first half of $[0,T]$, and this latter effect can outweigh the increased volatility risk to some extend. Therefore, it is beneficial for agent~1 to delay selling in the central part of the time interval $[0,T]$ and to compensate by accelerating the strategy  toward the end. This effect leads to the  increase of the intermediate asset position $X_1^*(1)$ as observed in Figure~\ref{alphasigma2 fig1}. When $\alpha\sigma^2$ increases even further, the increase in volatility risk becomes dominant, and so $X_1^*(1)$ starts to decrease again. 
}\label{alphasigma explain fig}
\end{figure}

\bigskip

\hspace{-0.7cm}
\begin{minipage}[b]{8cm}
\centering
\includegraphics[width=5cm]{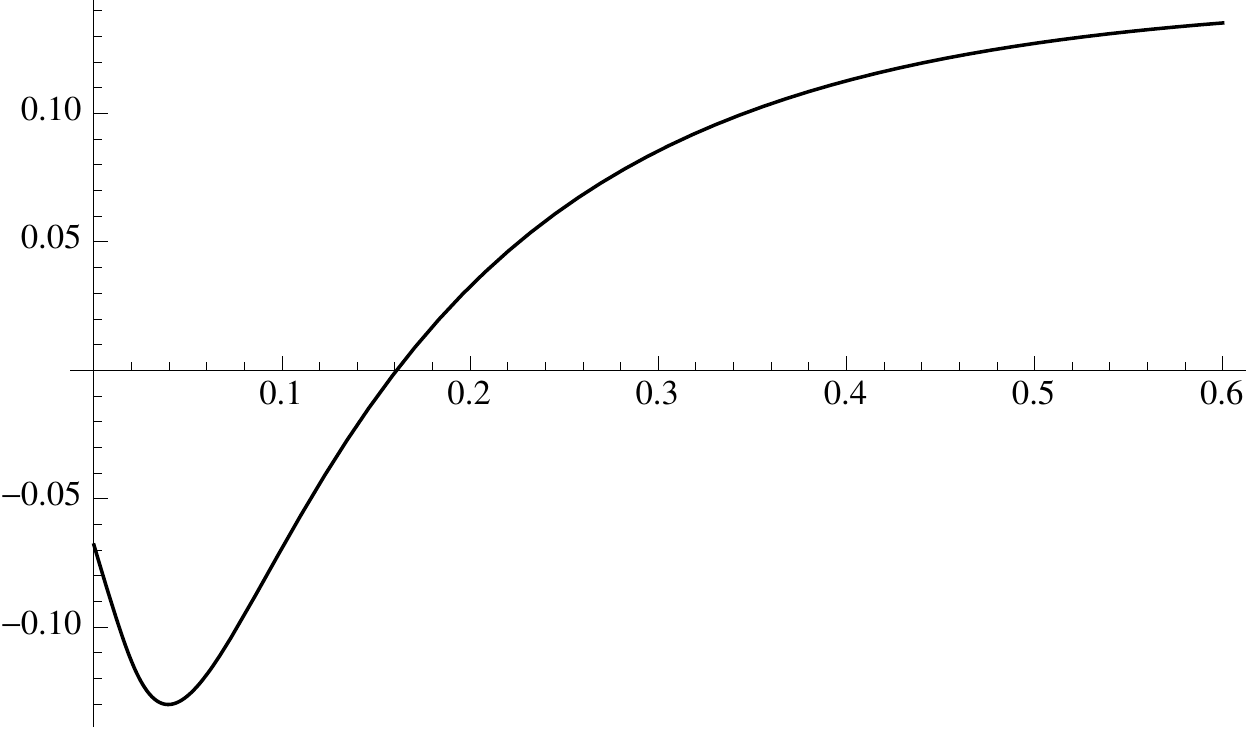}
\captionof{figure}{ $X^*_1(1)$  as a function of $\lambda$ for $x_1=0.2$, $x_2=4$, $T=2$, $\alpha\sigma^2=1$, and $\gamma=0.3$.}
\label{lambda fig}
\end{minipage}
\qquad
\begin{minipage}[b]{8cm}
\centering
\includegraphics[width=5cm]{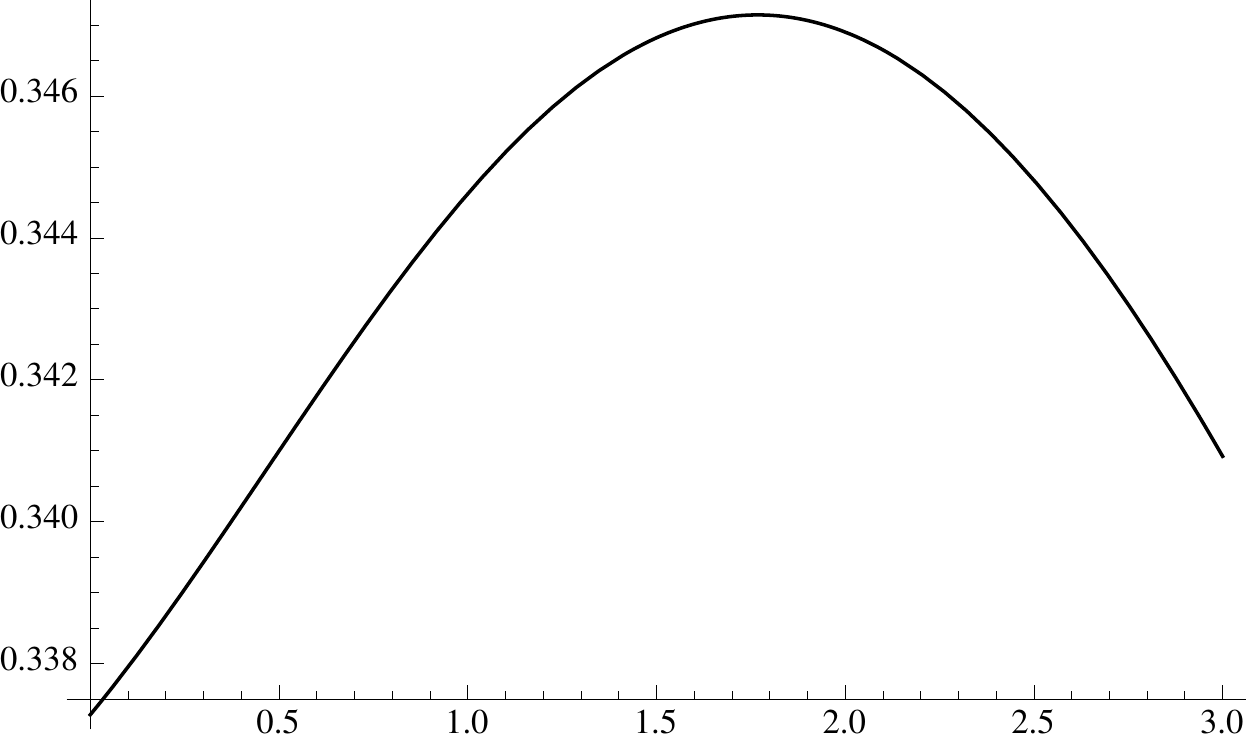}
\captionof{figure}{ $X^*_1(1)$ as a function of $\gamma$ for $x_1=0.86$, $x_2=0.28$, $T=2$, $\alpha\sigma^2=1$, and $\lambda=1$. }
\label{gamma fig}
\end{minipage}

\section{Nash equilibrium with infinite time horizon}
\label{infinite-horizon section}

Now, we consider  mean-variance optimization and CARA utility maximization for an infinite time horizon $[0,\infty)$.  Financially, this problem corresponds to a situation in which none of the agents faces a material time constraint. To simplify the discussion, we assume from the beginning that the drift $b(\cdot)$ vanishes identically. Then, the  unaffected price process is given by $
S^0(t)=S_0+\sigma W(t) $ for $t\ge0$. Here, we need to assume  that $\sigma\neq0$.
 If only one agent is active, we are in the situation of \citet{SchiedSchoeneborn}, where the problem of maximizing the expected utility of revenues is discussed for an infinite time horizon. As discussed there, a strategy $(X(t))_{t\ge0}$ should satisfy the following conditions of admissibility so that the utility-maximization problem is well-defined for a single agent:\\
 $\bullet$  $X$ is adapted to the filtration $(\cF_t)_{t\ge0}$;\\
$\bullet$  $X$ is absolutely continuous in the sense that $X(t)=X(0)+\int_0^t\dot X(s)\,ds$ for some progressively measurable process $(\dot X(t))_{t\ge0}$ for which
\begin{align}\label{infinite horizon admissibility 1}
\int_0^\infty(\dot X(t))^2\,dt<\infty\qquad\text{$\bP$-a.s.;}
\end{align}
$\bullet$  $X$ is bounded and satisfies\begin{align}\label{infinite horizon admissibility 2}
\bE\Big[\,\int_0^\infty X(t)^2\,dt\,\Big]<\infty\qquad\text{and}\qquad \lim_{t\ua\infty}(X(t))^2t\log\log t=0\text{ $\bP$-a.s.}
\end{align}
The class of all strategies that are admissible in this sense and satisfy $X(0)=x$ for given $x\in\bR$ will be denoted by $\cX(x,\infty)$. As before, we denote by $\cX_{\text{det}}(x,\infty)$ the subclass of all deterministic strategies in $\cX(x,\infty)$. 
When the admissible strategy $X$ is used, the affected price process is 
$$S^X(t)=S^0(t)+\gamma(X(t)-X(0))+\lambda\dot X(t).
$$
It is shown in \citet[Section 3.1]{SchiedSchoeneborn} that the total revenues of $X\in\cX(x,\infty)$ are $\bP$-a.s. well-defined as the limit
$$\cR(X):=-\lim_{T\ua\infty}\int_0^T\dot X(t)S^X(t)\,dt=xS_0-\frac\gamma2x^2+\sigma\int_0^\infty X(t)\,dW(t)-\lambda\int_0^\infty(\dot X(t))^2\,dt
$$
(see also Lemma~\ref{infinite horizon revenues lemma}  below). Moreover, for $\alpha>0$ and $u_\alpha$ as in~\eqref{ualpha}, the unique strategy that maximizes the expected utility $\bE[\,u_\alpha(\cR(X))\,]$ over $X\in\cX(x,\infty)$ is given by 
$$
X^*_0(t)=x\exp\Big(- t\sqrt{\frac{\alpha\sigma^2}{2\lambda}}\Big),\qquad t\ge0;
$$
see Corollary 4.4 in \citet{SchiedSchoeneborn}. As $\cR(X)$ is a Gaussian random variable for $X\in\cX_{\text{det}}(x,\infty)$, one sees that 
$$\bE[\,u_\alpha(\cR(X))\,]=\frac1\alpha\Big(1-e^{-\alpha \bE[\,\cR(X)\,]+\frac{\alpha^2}2\var(\cR(X))}\Big),\qquad X\in\cX_{\text{det}}(x,\infty),
$$
and so $X^*_0$ also maximizes the  mean-variance functional $\bE[\,\cR(X)\,]-\frac\alpha2\var(\cR(X))$ over $X\in \cX_{\text{det}}(x,\infty)$.

When  $n$ investors apply admissible strategies $X_1,X_2,\dots, X_n$, the affected price $S^{X_1,\dots,X_n}(t)$ is again given by~\eqref{affected price process n players eq}, as in the case of a finite time horizon. 
It will follow from Lemma~\ref{infinite horizon revenues lemma}
 below that the admissibility of $X_1,X_2,\dots, X_n$ guarantees that the following limit exists $\bP$-a.s.:
$$
\cR(X_i|\bm X_{-i}):=-\lim_{T\ua\infty}\int_0^T\dot X_i(t)S^{X_1,\dots,X_n}(t)\,dt.
$$
The  Nash equilibria for mean-variance optimization and CARA utility maximization can now be defined by taking $T=\infty$ in Definition~\ref{Nash eq def}. Here is our result on the existence and uniqueness of   Nash equilibria.


\begin{theorem}\label{Nash eq exist thm infinite horizon}Suppose that one of the following two conditions holds:
\begin{enumerate}
\item \label{Nash eq exist thm infinite horizon (a)}$n\in\bN$ is arbitrary and $\alpha_1=\cdots=\alpha_n=\alpha>0$;
\item $n=2$ and $\alpha_1$ and  $\alpha_2$ are distinct and strictly positive.\label{Nash eq exist thm infinite horizon (b)}
\end{enumerate}
Then, for all $x_1,\dots, x_n\in\bR$,  there exists a unique Nash equilibrium $(X^*_1,\dots,X^*_n)$ for mean-variance optimization with infinite time horizon. Moreover,  $(X^*_1,\dots,X^*_n)$ is also a Nash equilibrium for CARA utility maximization with infinite time horizon.

In case~\ref{Nash eq exist thm infinite horizon (a)}, the optimal strategies are given by  
\begin{align}\label{Nash eq exist thm infinite horizon strategy}
X^*_i(t)=(x_i-\bbar x_n)e^{\theta_-t}+\bbar x_ne^{\rho_-t},
\end{align}
where again $\bbar x_n=\frac1n\sum_{j=1}^nx_j$ and $\rho_-$ and $\theta_-$ are as in~\eqref{thetarho+-}.

In case~\ref{Nash eq exist thm infinite horizon (b)}, the fourth-order equation
$$
\tau^4-\frac{2\gamma}{3\lambda}\tau^3-\frac{\gamma^2+2\lambda\sigma^2(\alpha_1+\alpha_2)}{3\lambda^2}\tau^2+\frac{\sigma^4\alpha_1\alpha_2}{3\lambda^2}=0
$$
has precisely  two distinct  roots, $\tau_1$, $\tau_2$, in $(-\infty,0)$, and the equilibrium strategies $X^*_1(t)$ and $X^*_2(t)$ are linear combinations of the exponential functions $e^{\tau_1t}$ and $e^{\tau_2t}$.
\end{theorem}


On the one hand, the structure of the equilibrium strategies for an infinite time horizon appears to be simpler than for the finite-time situation. On the other hand, the assumptions of Theorem~\ref{Nash eq exist thm infinite horizon} are more restrictive than those of  Theorem~\ref{Nash Eq exist thm}. More restrictive assumptions are needed, because   all solutions $X_1(t),\dots, X_n(t)$ of the system~\eqref{ode system 1 thm} are linear combinations of exponential functions and thus can only take the limits $\pm\infty$ and $0$ for $t\ua\infty$. We must single out those with limit $0$. To this end, we cannot apply standard results on the existence of solutions for boundary value problems on noncompact intervals such as those in \citet{Cecchi}, where it is required that the possible boundary values at $t=\infty$ include the full space $\bR^n$. Instead, we   show here that the eigenspaces associated with the negative eigenvalues of a certain nonsymmetric matrix $M$ are sufficiently rich. For $n>2$, we are only able to understand these eigenspaces when $\alpha_1=\cdots=\alpha_n$.

\begin{remark}\label{convergence rem} In the situation of part~\ref{Nash eq exist thm infinite horizon (a)} of Theorem~\ref{Nash eq exist thm infinite horizon}, consider the corresponding  Nash equilibrium $X_1^{(T)},\dots, X_n^{(T)}$ for the finite time interval $[0,T]$ as constructed in Theorem~\ref{alpha n=alpha b=0 thm}. Then, we conclude  from~\eqref{linear comb rep} and~\eqref{ci(rho pm)} that 
$$
\lim_{T\ua\infty}X_i^{(T)}(t)=X_i^*(t),\qquad\text{for $i=1,\dots, n$ and  $t\ge0$,}
$$
where $X_i^*$ is as in~\eqref{Nash eq exist thm infinite horizon strategy}.
\end{remark}


Let us finally discuss some qualitative properties of the Nash equilibrium in part~\ref{Nash eq exist thm infinite horizon (a)} of Theorem~\ref{Nash eq exist thm infinite horizon}.  \citet{Carlinetal} and \citet{SchoenebornSchied} study, among other things,  whether the liquidation of a large block of shares by agent~1 leads either to  predatory trading or liquidity provision 
by the other agents if these all have zero initial capital (i.e., $x_i=0$ for $i\neq1$).  Here, predatory trading refers to a strategy during which the asset is  shortened at the initial high price and then bought back later when the sell strategy of agent~1 has depreciated the asset price. This strategy is \lq\lq predatory" in the sense that the revenues it generates  for agent $i$ are made at the expense of agent~1. Liquidity provision refers to exactly the opposite strategy: agent $i$  acquires a long position by first  buying and later re-selling  some of the shares agent~1 is liquidating. It can hence be seen as a cooperative behavior on behalf of agent $i$. Both \citet{Carlinetal}
and \citet{SchoenebornSchied} consider risk-neutral agents who need to close their positions in finite time. In  \citet{Carlinetal}, all agents face the same time constraint. In this case, liquidity provision can only be observed if cooperation is enforced by repeating the game. \citet{SchoenebornSchied} admit a longer time horizon for agents $i=2,\dots,n$ than for agent~1 and find that this relaxation can lead to liquidity provision  for certain  parameter values without having to repeat the game. Our corresponding result  is Corollary~\ref{predatory cor} below. It states that, on an infinite time horizon,  both predatory trading and liquidity provision can occur, depending on the parameters of the model; see also Figure~\ref{PredatoryVSLiquidity Fig} for an illustration. Together with Remark~\ref{convergence rem},  Corollary~\ref{predatory cor} implies that  liquidity provision can also occur if all agents share the same time horizon $T$, provided that $T$ is sufficiently large. This fact that is markedly different from the risk-neutral case $\alpha=0$ considered in  \citet{Carlinetal}.


\begin{corollary}\label{predatory cor}In the situation of part~\ref{Nash eq exist thm infinite horizon (a)} of Theorem~\ref{Nash eq exist thm infinite horizon}, suppose that $\sum_{i=1}^nx_i>0$. Then an agent with $x_i=0$ engages in liquidity provision in the sense that $X^*_i(t)>0$ for all $t>0$, if and only if $\alpha\sigma^2\lambda >2\gamma^2$. When  $\alpha\sigma^2\lambda <2\gamma^2$ this agent engages in predatory trading, and for  $\alpha\sigma^2\lambda =2\gamma^2$ the agent does not trade at all.
\end{corollary}


Finally, we briefly discuss the behavior of equilibrium strategies as a function of the number $n$ of agents active in the market. \citet{Maug} discuss the following  two hypotheses and analyze their validity for a large data set of block executions by large insiders:\\
{\bf Hypothesis 1:} \lq\lq Trade duration decreases if several insiders compete for exploiting the same long-lived information."\\
{\bf Hypothesis 2:} \lq\lq Trade duration increases if several insiders trade simultaneously in the same direction for liquidity reasons."\\
In the situation of part~\ref{Nash eq exist thm infinite horizon (a)} of our Theorem~\ref{Nash eq exist thm infinite horizon},  the effective trade duration can be both increasing or decreasing in $n$,  or even lack monotonicity entirely; see Figure~\ref{n gamma fig}.  Here, the effective trade duration is defined as the time until a certain high percentage of the initial inventory has been liquidated. So both hypotheses from  \citet{Maug}  are compatible with risk-averse agents in an Almgren--Chriss setting.

\begin{figure}[h]
\centering
\centering\includegraphics[width=7.2cm]{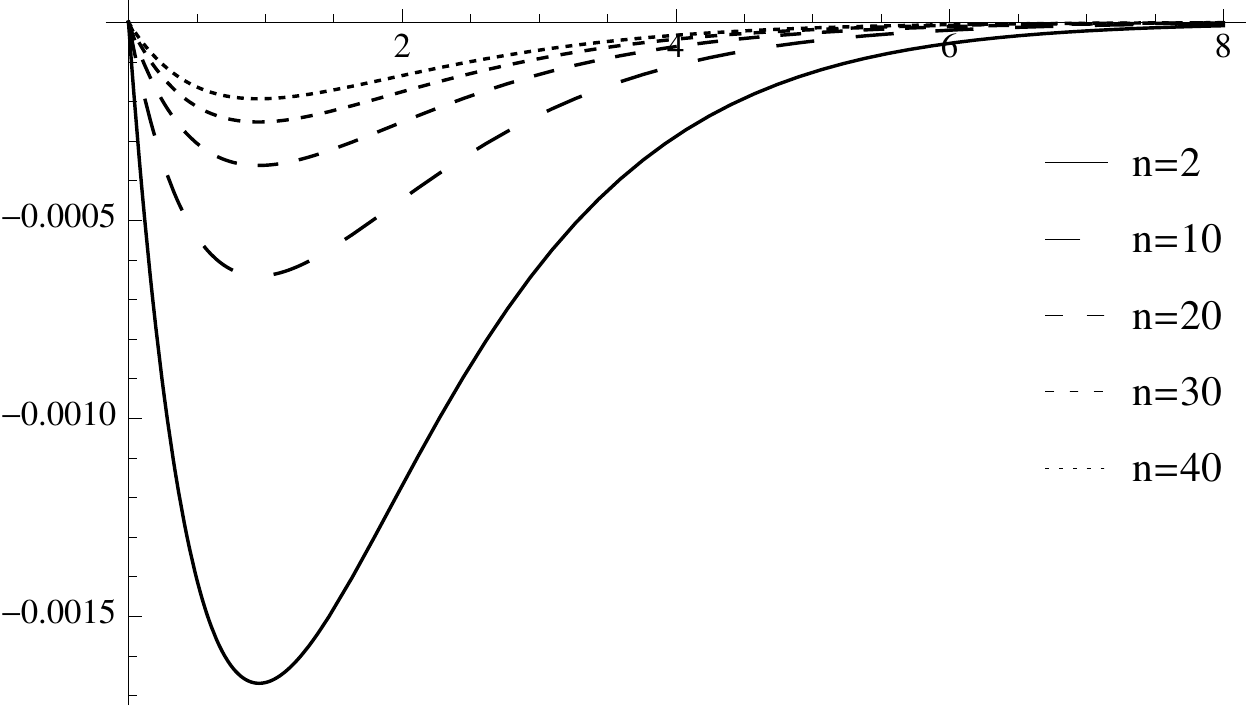}\qquad\qquad\qquad
\centering\includegraphics[width=7.2cm]{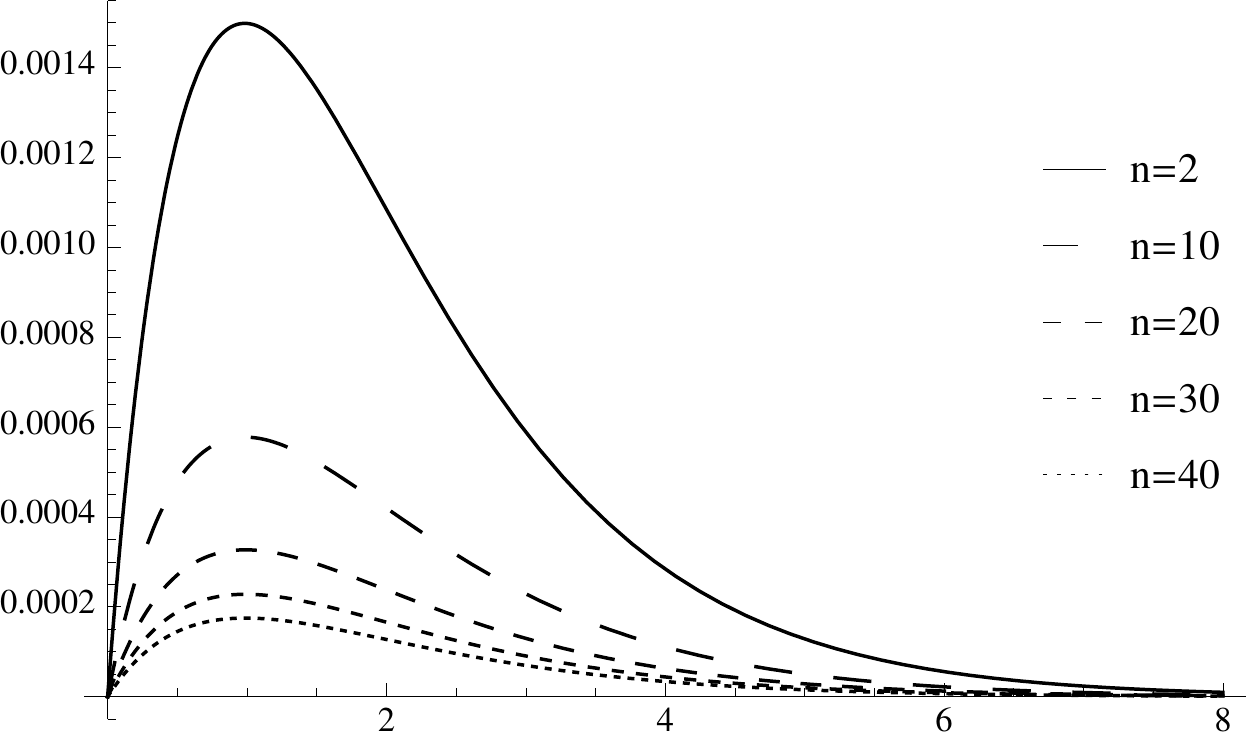}
\caption{Strategies $X^*_i(t)$ for $\lambda=0.15$ (left) and $\lambda=0.16$ (right) for various choices of $n$ and for $x_i=0$, $\sum_{j=1}^nx_j=1$, $\gamma=0.16$, and $\alpha\sigma^2=0.33$.}\label{PredatoryVSLiquidity Fig}
\end{figure}

\begin{figure}[h]
\centering
\centering\includegraphics[width=7.2cm]{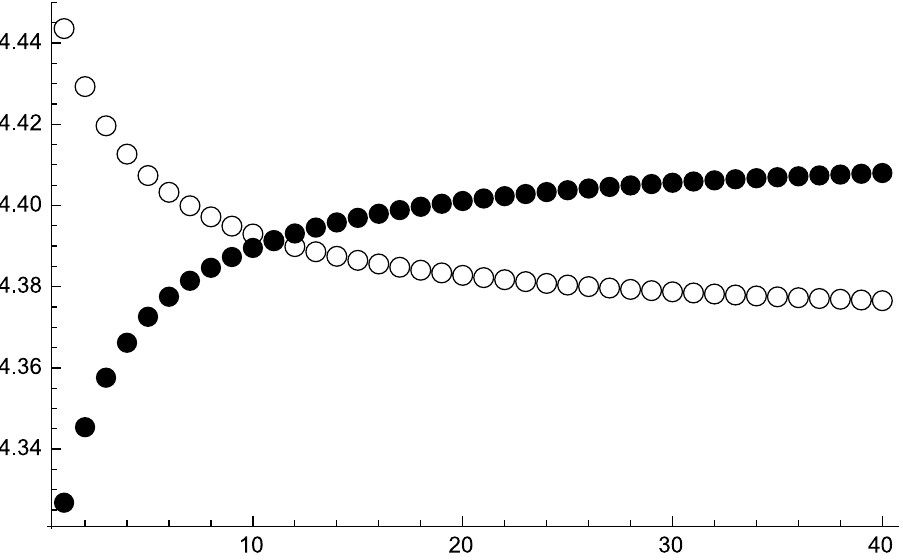}\qquad\qquad\qquad
\centering\includegraphics[width=7.2cm]{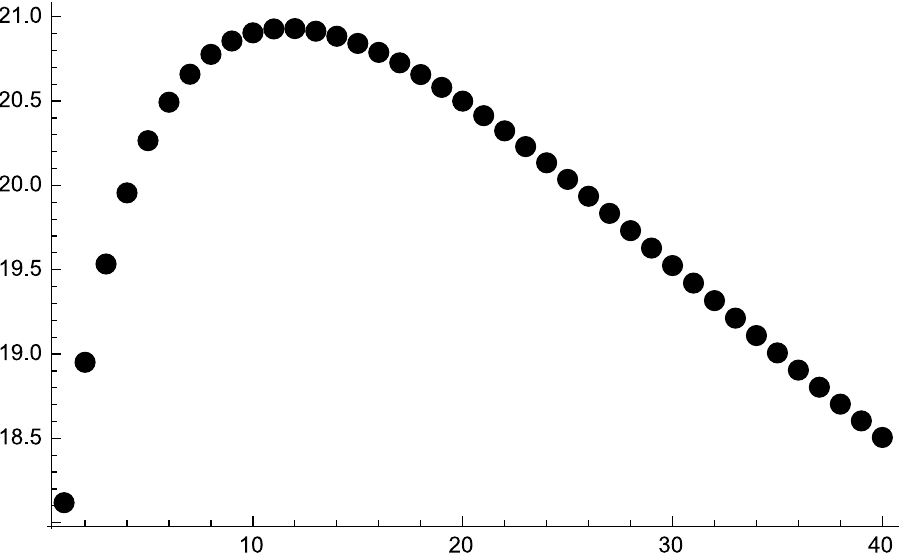}
\caption{Effective liquidation time of $X^*_1$, defined as the time until 99\% of the initial inventory $x_1$ have been liquidated, plotted  as a function of $n\in\{1,\dots,40\}$. In the left-hand panel, we we see that monotonicity in $n$ can be reversed by a very small change in $\gamma$; we took $\gamma=0.155$ (circles) and $\gamma=0.16$ (bullets) with $x_1=1$, $\sum_{i=1}^nx_i=3.5$, $\lambda=0.15$, and $\alpha\sigma^2=0.33$. In the right-hand panel, we observe that the effective liquidation time need not be monotone in $n$; here we chose $x_1=5$, $\sum_{i=1}^nx_i=10$, $\lambda=2$, $\gamma=0.1$,  and $\alpha\sigma^2=0.33$.}\label{n gamma fig}
\end{figure}

\section{Proofs}\label{Proofs Section}

\subsection{Proofs for a finite time horizon}

Let admissible strategies $X_i\in\cX(x_i,T)$ be given and write $\bm X_{-i}:=\{X_1,\dots,X_{i-1},X_{i+1},\dots,X_n\}$ for $i=1,\dots, n$. 
For $Y\in\cX(y,T)$,
we note first that, after integrating by parts, 
\begin{align*}
\cR(Y|\bm X_{-i})&= yS_0-\frac\gamma2y^2+\int_0^TY(t)\Big(b(t)+\gamma\sum_{j\neq i}\dot X_j(t)\Big)\,dt\\
&\qquad-\lambda\sum_{j\neq i}^n\int_0^T\dot Y(t)\dot X_j(t)\,dt-\lambda\int_0^T\dot Y(t)^2\,dt+\sigma\int_0^TY(t)\,dW(t).
\end{align*}
When all $X_i$ and $Y$ are deterministic, it follows that 
\begin{align}\label{mean-variance = action functional eq}
\bE[\,\cR(Y|\bm X_{-i})\,]
-\frac{\alpha_i}2\,\var(\cR(Y|\bm X_{-i}))=c+\int_0^T\cL^i(t,Y(t),\dot Y(t)|\bm X_{-i})\,dt,
\end{align}
where $c=yS_0-\frac\gamma2y^2$ and the Lagrangian $\cL^i$ is given by
\begin{align}\label{Lagrangian}
\cL^i(t,q,p|\bm X_{-i})=q\Big(b(t)+\gamma\sum_{j\neq i}\dot X_j(t)\Big)-\frac{{\alpha_i}\sigma^2}2q^2-\lambda p\Big(\sum_{j\neq i}\dot X_j(t)+ p\Big).
\end{align}


\begin{lemma}\label{Equilibrium uniqueness lemma}In the context of Theorem~\ref{Nash Eq exist thm}, there exists at most one Nash equilibrium for mean-variance optimization.
\end{lemma}

\begin{proof}We assume by way of contradiction that $X^0_1,\dots,X^0_n$ and $X^1_0,\dots, X^1_n$ are two distinct Nash equilibria with $X^k_i\in\cX(x_i,T)$ for $i=1,\dots, n$ and $k=0,1$. For $\beta\in[0,1]$, let $X^\beta_i:=\beta X^1_i+(1-\beta)X^0_i$ and define 
$$f(\beta):=\sum_{i=1}^n\int_0^T\Big(\cL^i(t,X_i^\beta(t),\dot X_i^\beta(t)|{\bm X}^0_{-i})+\cL^i(t,X_i^{1-\beta}(t),\dot X_i^{1-\beta}(t)|{\bm X}^1_{-i})\Big)\,dt.
$$
By assumption, the strategy $X_i^k$ maximizes the functional $Y\mapsto 
\int_0^T\cL^i(t,Y(t),\dot Y(t)|{\bm X}^k_{-i})\,dt$ within the class $\cX_{\text{det}}(x_i,T)$  for $k=0,1$. We therefore must have $f(\beta)\le f(0)$ for $\beta>0$, which implies that
\begin{align}\label{f'(0) contradiction eq}
\frac d{d\beta}\Big|_{\beta=0+}f(\beta)\le0.
\end{align}
On the other hand, by  interchanging differentiation and integration, which is permitted due to our assumptions on admissible strategies and due to the linear-quadratic form of the Lagrangian, a short computation shows that 
\begin{align*}
\lefteqn{\frac d{d\beta}\Big|_{\beta=0+}f(\beta)}\\
&=\sum_{i=1}^n\int_0^T\bigg[\gamma (X^1_i(t)-X^0_i(t))\sum_{j=1}^n(\dot X_j^0(t)-\dot X_j^1(t))-\gamma (X^1_i(t)-X^0_i(t))(\dot X_i^0(t)-\dot X^1_i(t))\\
&\quad+{\alpha_i}\sigma^2 (X^1_i(t)-X^0_i(t))^2+\lambda(\dot X^1_i(t)-\dot X^0_i(t))\sum_{j=1}^n(\dot X^1_j(t)-\dot X_j^0(t)) +\lambda (\dot X^0_i(t)-\dot X^1_i(t)) ^2\bigg]\,dt.\end{align*}
We note next that 
$$\int_0^T(X^1_i(t)-X^0_i(t))(\dot X^0_i(t)-\dot X^1_i(t))\,dt=\frac12(X^1_i(T)-X^0_i(T))^2-\frac12(X^1_i(0)-X^0_i(0))^2=0.
$$
Moreover, by the same argument,
$$\int_0^T(X^1_i(t)-X^0_i(t))(\dot X^0_j(t)-\dot X^1_j(t))\,dt=-\int_0^T(X^1_j(t)-X^0_j(t))(\dot X^0_i(t)-\dot X^1_i(t))\,dt,
$$
and hence
$$\sum_{i=1}^n\sum_{j=1}^n\int_0^T(X^1_i(t)-X^0_i(t))(\dot X^0_j(t)-\dot X^1_j(t))\,dt=0.
$$
It follows that 
\begin{align*}
{\frac d{d\beta}\Big|_{\beta=0+}f(\beta)}=\int_0^T\bigg[{\alpha_i}\sigma^2\sum_{i=1}^n(X^1_i(t)-X^0_i(t))^2+\lambda\sum_{i=1}^n (\dot X^0_i(t)-\dot X^1_i(t)) ^2+\lambda \Big(\sum_{i=1}^n(\dot X^0_i(t)-\dot X^1_i(t))\Big)^2\bigg]\,dt,
\end{align*}
which is strictly positive because the two Nash equilibria $X^0_1,\dots,X^0_n$ and $X^1_0,\dots, X^1_n$ are distinct. But strict positivity contradicts~\eqref{f'(0) contradiction eq}. 
\end{proof}


\begin{lemma}\label{one-agent optimization problem lemma}For $i=1,\dots, n$ there exists at most one maximizer in $\cX_{\text{det}}(y,T)$ of the functional $Y\mapsto 
\int_0^T\cL^i(t,Y(t),\dot Y(t)|\bm X_{-i})\,dt$. If, moreover, $X_1,\dots, X_n\in C^2[0,T]$, then there exists a unique maximizer   $Y^*\in\cX_{\text{det}}(y,T)\cap C^2[0,T]$, which is given as the unique solution of the two-point boundary value problem
$$\left\{
\begin{array}{l} {\alpha_i}\sigma^2Y(t)
-2\lambda\ddot Y(t)=b(t)+\gamma\sum_{j\neq i}\dot X_j(t)+\lambda \sum_{j\neq i}\ddot X_j(t),\\
Y(0)=y,\ 
Y(T)=0.
\end{array}
\right.
$$
\end{lemma}

\begin{proof}
It follows from the strict concavity of the Lagrangian $\cL^i$ and the convexity of  the set $\cX_{\text{det}}(y,T)$  that there can be at most one maximizer in $\cX_{\text{det}}(y,T)$. 

Now, we show the existence of a maximizer under the additional assumption $X_1,\dots, X_n\in C^2[0,T]$.  Under this assumption, we may formulate the Euler--Lagrange equation  $\cL^i_q(t,Y(t),\dot Y(t)|\bm X_{-i})=\frac d{dt}\cL^i_p(t,Y(t),\dot Y(t)|\bm X_{-i})$, which for our specific Lagrangian becomes:
\begin{align}\label{first Euler Lagrange eqn}
{\alpha_i}\sigma^2Y(t)
-2\lambda\ddot Y(t)=b(t)+\gamma\sum_{j\neq i}\dot X_j(t)+\lambda \sum_{j\neq i}\ddot X_j(t).
\end{align}
 Denoting the right-hand side of~\eqref{first Euler Lagrange eqn} by $u(t)$,  the general solution of this second-order ODE is of the form
\begin{equation*}
 Y(t)=c_1 e^{-\kappa_i t}+c_2e^{\kappa_i t}-\frac1{4\lambda\kappa_i}\int_0^te^{\kappa_i(t-s)}u(s)\,ds+\frac1{4\lambda\kappa_i}\int_0^te^{-\kappa_i(t-s)}u(s)\,ds,
\end{equation*}
where $c_1$ and $c_2$ are constants and
$
 \kappa_i=\sqrt{{\alpha_i\sigma^2}/{2\lambda}}$.
It is clear that the two constants $c_1$ and $c_2$ can be uniquely determined by imposing the boundary conditions $Y(0)=y$ and $Y(T)=0$. From now on, let $Y^*\in\cX_{\text{det}}(y,T)\cap C^2[0,T]$ denote the corresponding solution. We will now verify that $Y^*$ is indeed a maximizer of our problem. To this end, let $Y\in\cX_{\text{det}}(y,T)$ be arbitrary. Using first the concavity of  $(q,p)\mapsto \cL^i(t,q,p|\bm X_{-i})$ and then the fact that $Y^*$ solves the Euler--Lagrange equation, we get 
\begin{align*}
\lefteqn{\cL^i(t,Y^*(t),\dot Y^*(t)|\bm X_{-i}))-\cL^i(t,Y(t),\dot Y(t)|\bm X_{-i})}\\
&\ge \cL^i_q(t,Y^*(t),\dot Y^*(t)|\bm X_{-i}))(Y^*(t)-Y(t))+ \cL^i_p(t,Y^*(t),\dot Y^*(t)|\bm X_{-i}))(\dot Y^*(t)-\dot Y(t))\\
&=\Big(\frac d{dt}\cL^i_p(t,Y^*(t), \dot Y^*(t)|\bm X_{-i})\Big)(Y^*(t)-Y(t))+ \cL^i_p(t,Y^*(t),\dot Y^*(t)|\bm X_{-i}))(\dot Y^*(t)-\dot Y(t))\\
&=\frac d{dt}\Big(\cL^i_p(t,Y^*(t),\dot Y^*(t)|\bm X_{-i})(Y^*(t)-Y(t))\Big).
\end{align*}
Therefore, 
\begin{align*}\lefteqn{\int_0^T\cL^i(t,Y^*(t),\dot Y^*(t)|\bm X_{-i})\,dt-\int_0^T\cL^i(t,Y(t),\dot Y(t)|\bm X_{-i})\,dt}\\
&\ge \int_0^T\frac d{dt}\Big(\cL^i_p(t,Y^*(t),\dot Y^*(t)|\bm X_{-i})(Y^*(t)-Y(t))\Big)\,dt= 0,
\end{align*}
where in the final step we have used that $Y^*(0)=Y(0)$ and $Y^*(T)=Y(T)$.  
This proves the lemma.
\end{proof}


\begin{proof}[Proof of Theorem~\ref{Nash Eq exist thm}] According to Lemma~\ref{Equilibrium uniqueness lemma}, there exists at most one Nash equilibrium. We will now show that there exists a Nash equilibrium $X^*_1,\dots,X^*_n$ such that each strategy $X^*_i$ belongs to $\cX_{\text{det}}(x_i,T)\cap C^2[0,T]$. By Lemma~\ref{one-agent optimization problem lemma}, each strategy $X^*_i$ must then be a solution of the second-order differential equation 
\begin{equation}\label{ode system 1}
 {\alpha_i}\sigma^2X_i(t)
-2\lambda\ddot X_i(t)=b(t)+\gamma\sum_{j\neq i}\dot X^*_j(t)+\lambda \sum_{j\neq i}\ddot X^*_j(t), \end{equation}
with boundary conditions 
\begin{align}\label{ode system 1 boundary conditions}
\text{$X_i(0)=x_i$ and $X_i(T)=0$.}
\end{align}
 We can clearly combine the $n$ differential equations~\eqref{ode system 1} into a system of $n$ coupled second-order linear ordinary differential equations for the vector $\bm X^*:=(X^*_1,\dots,X^*_n)^\top$. 
 It follows again from Lemma~\ref{one-agent optimization problem lemma} that every $C^2$-solution of the system~\eqref{ode system 1},~\eqref{ode system 1 boundary conditions}
 is a Nash equilibrium. Therefore the assertion of the theorem will follow if we can  show the existence of a $C^2$-solution to the $n$-dimensional two-point boundary value problem~\eqref{ode system 1},~\eqref{ode system 1 boundary conditions}.

 By introducing the auxiliary function $\bm Y(t)$ for the derivative $ \dot {\bm X}(t)$, by letting $\bm b(t)$ be the vector with all components equal to $b(t)$, and  by defining the $n\times n$ matrices $A:=\sigma^2\text{diag} (\alpha_1,\dots,\alpha_n)$, the identity matrix $I=\text{diag}(1,\dots,1)$, and the matrix $J$ with all entries equal to one, the system~\eqref{ode system 1} can be re-written as follows:
 \begin{align}\label{ODE system 2}
\begin{pmatrix}A&-\gamma(J-I)\\
 0&I
 \end{pmatrix}\begin{pmatrix}\bm X(t)\\\bm Y(t) \end{pmatrix}-\begin{pmatrix}0&\lambda(J+I)\\ I&0 \end{pmatrix}
\begin{pmatrix}\dot{\bm X}(t)\\\dot{\bm Y}(t) \end{pmatrix}
=\begin{pmatrix} \bm b(t)\\\bm 0\end{pmatrix}.
\end{align} 
Clearly, $\bm X$ is a  $C^2$-solution of~\eqref{ode system 1} if and only if ${\bm X\choose \dot{\bm X}}$  is a $C^1$-solution of~\eqref{ODE system 2}. In particular, every $C^1$-solution of~\eqref{ODE system 2} with boundary conditions~\eqref{ode system 1 boundary conditions} yields a Nash equilibrium.

Now consider the homogeneous system~\eqref{ODE system 2},~\eqref{ode system 1 boundary conditions} with $\bm b(t)=0$ and initial values $x_1=\cdots=x_n=0$. 
The corresponding boundary condition can be written as 
\begin{align}\label{ODE system 2 boundary cond}
(\bm X(0),\bm Y(0),\bm X(T),\bm Y(T))^\top\in V,
\end{align}
where $V\subset\bR^{4n}$ is the $2n$-dimensional linear space
$$V=\big\{(\bm x_0,\bm y_0,\bm x_1,\bm y_1)^\top \in\bR^{4n}\,|\,\bm x_0=\bm x_1=\bm0\big\}.
$$ 
It is clear that ${\bm X\choose {\bm Y}}={\bm 0\choose\bm 0}$ is a solution. In fact, this trivial solution is the only solution as every solution must be a Nash equilibrium, and Nash equilibria are unique by Lemma~\ref{Equilibrium uniqueness lemma}. It therefore follows from the general theory of linear boundary value problems for systems of ordinary differential equations that the two-point boundary value problem~\eqref{ODE system 2},~\eqref{ODE system 2 boundary cond} has a unique $C^1$-solution for every continuous $\bm b:[0,T]\to\bR^n$ (and in fact for every continuous $\bR^{2n}$-valued function substituting ${\bm b(t)\choose\bm0}$ on the right-hand side of~\eqref{ODE system 2}); see \citet[(9.22), p. 189]{Kurzweil}. Using this fact, we let  ${\bm X^0\choose {\bm Y}^0}$ be the solution of~\eqref{ODE system 2},~\eqref{ODE system 2 boundary cond} when $\bm b(t)$ in~\eqref{ODE system 2} is replaced by $\bm b^0(t)=(b_1^0(t),\dots,b^0_n(t))$ for
$$b^0_i(t)=b(t)+\frac{T-t}T\alpha_i\sigma^2x_i+\frac\gamma T\sum_{j\neq i}x_j.
$$
One then  checks that 
$$X_i^*(t):=X^0_i(t)+\frac{T-t}Tx_i,\qquad i=1,\dots, n,
$$
solves~\eqref{ode system 1},~\eqref{ode system 1 boundary conditions} and is thus the desired Nash equilibrium.
\end{proof}


\begin{remark}\label{matrix rm}For $A:=\sigma^2\text{diag} (\alpha_1,\dots,\alpha_n)$, the identity matrix $I$, and the matrix $J$ with all entries equal to one, let
$$M :=\begin{pmatrix}0&\lambda(J+I)\\ I&0 \end{pmatrix}^{-1}\begin{pmatrix}A&-\gamma(J-I)\\
 0&I
\end{pmatrix}.
$$
and 
$$\bm f(t):=-\begin{pmatrix}0&\lambda(J+I)\\ I&0 \end{pmatrix}^{-1}\begin{pmatrix} \bm b(t)\\\bm 0\end{pmatrix}.
$$ 
With this notation and $\bm Z(t):=(\bm X(t),\bm Y(t))^\top$, the system~\eqref{ODE system 2} can now be written as 
\begin{equation}\label{ODE formula}
\dot{\bm Z}(t)=M\bm Z(t)+\bm f(t).
\end{equation}
 Note that $J^2=nJ$ and that hence $(J+I)(I-\frac1{n+1}J)=I=(I-\frac1{n+1}J)(J+I)$. It follows that 
$$\begin{pmatrix}0&\lambda(J+I)\\ I&0 \end{pmatrix}^{-1}=\begin{pmatrix}0&I\\\frac1\lambda(I-\frac1{n+1}J)&0\end{pmatrix},
$$
and hence that
\begin{equation}\label{matrix M}M=\begin{pmatrix}0&I\\
\frac1\lambda(A-\frac1{n+1}JA)&\frac\gamma\lambda(I-\frac2{n+1}J)\end{pmatrix}.
\end{equation}
\end{remark}


\begin{proof}[Proof of Corollary~\ref{CARA Nash Eq exist cor}] Let $X^*_1,\dots,X^*_n$ be the unique Nash equilibrium for mean-variance optimization as constructed in Theorem~\ref{Nash Eq exist thm}. When ${\bm X}^*_{-i}=\{X^*_1,\dots,X^*_{i-1},X^*_{i+1},\dots,X^*_n\}$ is fixed, the $i^{\text{th}}$ agent perceives 
\begin{align*}
S^{{\bm X}^*_{-i}}(t):=S^0(t)+\gamma\sum_{j\neq i}(X_j(t)-X_j(0))+\lambda\sum_{j\neq i}\dot X_j(t),\qquad t\in[0,T],
\end{align*}
as \lq\lq unaffected" price process. It is of the form 
$$S^{{\bm X}^*_{-i}}(t)=S_0+\sigma W(t)+\int_0^tb^i(s)\,ds
$$
for a deterministic and continuous function $b^i:[0,T]\to\bR$.  As the process $S^{{\bm X}^*_{-i}}$ has independent increments and $S^{{\bm X}^*_{-i}}_T$ has all exponential moments, i.e., $\bE\big[\,e^{\beta S^{{\bm X}^*_{-i}}_t}\,\big]<\infty$ for all $\beta\in\bR$ and $t\ge0$, it follows as in \citet[Theorem 2.1]{SchiedSchoenebornTehranchi} that for $\alpha_i>0$
$$\sup_{X\in\cX(x_i,T)}\bE[\,u_{\alpha_i}(\cR(X|\bm X_{-i}))\,]=\sup_{X\in\cX_{\text{det}}(x_i,T)}\bE[\,u_{\alpha_i}(\cR(X|\bm X_{-i}))\,].
$$
But for $\alpha_i>0$ and $X\in\cX_{\text{det}}(x_i,T)$ we have 
$$\bE[\,u_{\alpha_i}(\cR(X|\bm X_{-i}))\,]=\frac1{\alpha_i}\Big(1-e^{-\alpha_i\bE[\,\cR(X|\bm X_{-i})\,]
+\frac{\alpha_i^2}2\,\var(\cR(X|\bm X_{-i})}\Big),
$$
which shows that CARA utility maximization is equivalent to the maximization of the corresponding mean-variance functional. 
The corresponding result for $\alpha_i=0$ is obvious. \end{proof}


\begin{proof}[Proof of Corollary~\ref{drift cor}] Letting $\Sigma(t):=\sum_{j=1}^nX_j(t)$ and re-writing~\eqref{ode system 1 thm} yields 
 $$ {\alpha}\sigma^2X_i(t)+\gamma \dot X_i(t)
-\lambda\ddot X_i(t)=b(t)+\gamma\dot\Sigma(t)+\lambda \ddot \Sigma(t)
 $$
 and hence~\eqref{Xi b bvp}.
 Summing over $i$ then implies~\eqref{Sigma b bvp}.
\end{proof}


Now we prepare for the proof of Theorem~\ref{alpha n=alpha b=0 thm}.


\begin{lemma}\label{matrix M lemma}For $\alpha_1=\cdots=\alpha_n=\alpha>0$, the matrix $M$ from~\eqref{matrix M} has four real eigenvalues $\theta_+,\theta_-,\rho_+,\rho_-$ given by~\eqref{thetarho+-}.
Moreover, with $\bm 1\in\bR^n$ denoting the vector with all entries equal to $1$, the corresponding eigenspaces are given by 
\begin{align*}
E(\rho_\pm)=\text{\rm span}{\bm 1\choose \rho_\pm\bm1}  \qquad \text{and}\qquad E(\theta_\pm)=\bigg\{{\bm v\choose \theta_\pm\bm v}\,\Big|\,\bm v\in\bR^n,\ \bm v\perp\bm 1\bigg\}.
\end{align*}
\end{lemma}

\begin{proof}Let us write an arbitrary vector in $\bR^{2n}$ as ${\bm v_1\choose\bm v_2}$ for $\bm v_1,\bm v_2\in\bR^n$. By applying $M$ to ${\bm v_1\choose\bm v_2}$ we see that we must have $\bm v_2=\tau\bm v_1$ for ${\bm v_1\choose\bm v_2}$ to be an eigenvector with eigenvalue $\tau$. So let us consider vectors in $\bR^{2n}$ of the form ${\bm v\choose\tau\bm v}$ for $\bm v\in\bR^n$ and $\tau\in\bR$. The equation $M{\bm v\choose\tau\bm v}=\tau {\bm v\choose\tau\bm v}$ is equivalent to 
\begin{align}\label{eigenvalue equation}
\Big(\frac{\alpha\sigma^2}{\lambda}+\frac{\tau\gamma}{\lambda}\Big)\bm v-\frac{\alpha\sigma^2+2\tau\gamma}{{\lambda(n+1)}}J\bm v=\tau^2\bm v.
\end{align}
When $\bm v=\bm 1$, then $J\bm v=n\bm v$ and~\eqref{eigenvalue equation} becomes the quadratic equation 
\begin{align*}
\alpha\sigma^2+\gamma\tau-\frac{n(\alpha\sigma^2+2\gamma\tau)}{n+1}-\lambda\tau^2=0,
\end{align*}
which is solved for $\tau=\rho_+$ and $\tau=\rho_-$. When $\bm v\perp\bm1$, then $J\bm v=0$ and~\eqref{eigenvalue equation} becomes the quadratic equation 
\begin{align*}
\alpha\sigma^2+\gamma\tau-\lambda\tau^2=0,
\end{align*}
 which is solved for $\tau=\theta_+$ and $\tau=\theta_-$. As the eigenvectors found thus far span the entire space $\bR^{2n}$, the proof is complete.
\end{proof}


\begin{proof}[Proof of Theorem~\ref{alpha n=alpha b=0 thm}.]
It follows from Theorem~\ref{Nash Eq exist thm} and its proof that $X^*_1,\dots,X^*_n$ are obtained from the solutions of~\eqref{ODE formula} for $\bm f(t)=\bm0$. The general solution of this system is of the form $\bm Z(t)=e^{tM}\bm Z(0)$. By Lemma~\ref{matrix M lemma}, $M$ is diagonalizable and so every solution $\bm Z(t)$ must be a linear combination of exponential functions $e^{\tau t}$, where $\tau$ is an eigenvalue of $M$.  Another application of Lemma~\ref{matrix M lemma} thus implies that each $X^*_i$ can be represented as in~\eqref{linear comb rep}. One finally checks that for  $c_i(\theta_+),c_i(\theta_-),c(\rho_+),c(\rho_-)$ as in~\eqref{ci(rho pm)} the boundary conditions $X^*_i(0)=x_i$ and $X^*_i(T)=0$ are satisfied. That $\Sigma$ from~\eqref{Sigma general n eq} solves the two-point boundary problem~\eqref{Sigma b bvp} can be verified by a straightforward computation.\end{proof}


\begin{proof}[Proof of Corollary~\ref{n=2 Nash eq cor}] From~\eqref{Sigma general n eq} we have that  $\Sigma(t)=X_1^*(t)+X_2^*(t)$ is given by~\eqref{Sigma for n=2}. 
When letting $\Delta(t):=X^*_1(t)-X^*_2(t)$, we get from~\eqref{Xi b bvp} that $\Delta$ solves the two-point boundary value problem
$$\alpha\sigma^2\Delta(t)+\gamma\dot\Delta(t)-\lambda\ddot\Delta(t)=0,\qquad\Delta(0)=x_1-x_2,\ \Delta(T)=0.
$$
This boundary value problem is solved by~\eqref{Delta for n=2}. 
\end{proof}

\subsection{Proofs for an infinite time horizon}

\begin{lemma}\label{infinite horizon revenues lemma} For $X_i\in\cX(x_i,\infty)$ and $i=1,\dots, n$, the limit
$\cR(X_i|\bm X_{-i}):=-\lim_{T\ua\infty}\int_0^T\dot X_i(t)S^{X_1,\dots,X_n}(t)\,dt$
exists, is finite,  and is given by
\begin{equation*}\label{infinite horizon revenues formula}
\cR(X_i|\bm X_{-i})=x_iS_0-\frac\gamma2x_i^2+\sigma\int_0^\infty  \!\!X_i(t)\,dW(t)+\gamma\sum_{j\neq i}\int_0^{\infty}  \!\!X_i(t)\dot X_j(t)\,dt-\lambda\sum_{j=1}^n\int_0^{\infty} \!\!\dot X_i(t)\dot X_j(t)\,dt.\end{equation*}
\end{lemma}

\begin{proof}Integrating by parts yields
\begin{align*}\lefteqn{-\int_0^T\dot X_i(t)S^{X_1,\dots,X_n}(t)\,dt}\\
&\qquad=(x_i-X_i(T))S_0-X(T)W(T)+\sigma\int_0^T X_i(t)\,dW(t)-\frac\gamma2(X_i(T)-X_i(0))^2\\
&\qquad\qquad-\gamma\sum_{j\neq i}X_i(T)(X_j(T)-X_j(0))+\gamma\sum_{j\neq i}\int_0^{T} X_i(t)\dot X_j(t)\,dt-\lambda\sum_{j=1}^n\int_0^{T}\dot X_i(t)\dot X_j(t)\,dt.
\end{align*}
The assertion now follows by using the law of the iterated logarithm for $W$,~\eqref{infinite horizon admissibility 1},~\eqref{infinite horizon admissibility 2},   and the Cauchy--Schwarz inequality.
\end{proof}

Now let $X_i\in\cX(x_i,\infty)$, $i=1,\dots, n$, be given. 
As in~\eqref{mean-variance = action functional eq},~\eqref{Lagrangian}, we get that for $Y\in\cX(y,\infty)$,
\begin{align*}
\bE[\,\cR(Y|\bm X_{-i})\,]
-\frac{\alpha_i}2\,\var(\cR(Y|\bm X_{-i}))=c+\int_0^\infty\cL^i(t,Y(t),\dot Y(t)|\bm X_{-i})\,dt,
\end{align*}
where $c=yS_0-\frac\gamma2y^2$ and the Lagrangian $\cL^i$ is given by~\eqref{Lagrangian}.
Recall that we assume $\sigma^2>0$.


\begin{lemma}\label{one-agent optimization problem lemma infinite horizon}For $i=1,\dots, n$ and $\alpha_i>0$, the functional $Y\mapsto 
\int_0^\infty\cL^i(t,Y(t),\dot Y(t)|\bm X_{-i})\,dt$ has at most one maximizer in $\cX_{\text{det}}(y,\infty)$. If, moreover, $X_1,\dots, X_n$ belong to $ C^2[0,\infty)$ and are such that
\begin{align}\label{derivatives int condition for Euler Lagrange}
\int_0^\infty\bigg|\gamma\sum_{j\neq i}\dot X_j(t)+\lambda \sum_{j\neq i}\ddot X_j(t)\bigg|\,dt<\infty,
\end{align}
 then there exists a unique maximizer   $Y^*\in\cX_{\text{det}}(y,\infty)\cap C^2[0,\infty)$, which is given as the unique solution of the boundary value problem
\begin{equation}\label{coupled Euler-Lagrange inf horizon}
 {\alpha_i}\sigma^2Y(t)
-2\lambda\ddot Y(t)=\gamma\sum_{j\neq i}\dot X_j(t)+\lambda \sum_{j\neq i}\ddot X_j(t),\qquad
Y(0)=y,\ \lim_{t\ua\infty}Y(t)=0.
 \end{equation}
 Moreover, $Y$ satisfies $\int_0^\infty|\dot Y(t)|+|\ddot Y(t)|\,dt<\infty$.
\end{lemma}

\begin{proof}
It follows from the strict concavity of the Lagrangian $\cL^i$, the convexity of  the set $\cX_{\text{det}}(y,\infty)$, and the finiteness of the integral  $\int_0^\infty\cL^i(t,Y(t),\dot Y(t)|\bm X_{-i})\,dt$  that there can be at most one maximizer in $\cX_{\text{det}}(y,\infty)$. 

Now, we show the existence of a maximizer under the additional assumptions $X_1,\dots, X_n\in C^2[0,\infty)$ and~\eqref{derivatives int condition for Euler Lagrange}.  As noted in the proof of Lemma~\ref{one-agent optimization problem lemma}, the general solution of the Euler--Lagrange equation~\eqref{first Euler Lagrange eqn} is given by
\begin{equation}\label{ODE solution form rep infinite horizon}
 Y(t)=c_1 e^{-\kappa_i t}+c_2e^{\kappa_i t}-\frac1{4\lambda\kappa_i}\int_0^te^{\kappa_i(t-s)}u(s)\,ds+\frac1{4\lambda\kappa_i}\int_0^te^{-\kappa_i(t-s)}u(s)\,ds,
\end{equation}
where $u(t)=\gamma\sum_{j\neq i}\dot X_j(t)+\lambda \sum_{j\neq i}\ddot X_j(t)$, $c_1$ and $c_2$ are constants, and $
 \kappa_i=\sqrt{{\alpha_i\sigma^2}/{2\lambda}}>0$. One checks that~\eqref{derivatives int condition for Euler Lagrange}
 implies that $\int_0^te^{-\kappa_i(t-s)}u(s)\,ds\to0$ as $t\ua\infty$. Therefore, when letting
 \begin{align}\label{c2 eq}
 c_2:=\frac1{4\lambda\kappa_i}\int_0^\infty e^{-\kappa_is}u(s)\,ds
\end{align}
  and $c_1:=y-c_2$, one sees that the corresponding function $Y^*$ solves~\eqref{coupled Euler-Lagrange inf horizon}.
  
 Next, by~\eqref{coupled Euler-Lagrange inf horizon},~\eqref{ODE solution form rep infinite horizon}, and~\eqref{c2 eq}, $Y^*(t)$, $\dot Y^*(t)$, and $\ddot Y^*(t)$  are linear combinations of the following functions:
  $$u(t),\quad e^{-\kappa_i t},\quad \int_0^te^{-\kappa_i(t-s)}u(s)\,ds, \quad \int_t^\infty e^{\kappa_i(t-s)}u(s)\,ds.
  $$
 We have
  \begin{align*}\int_0^T\int_0^te^{-\kappa_i(t-s)}|u(s)|\,ds\,dt=\frac1{\kappa_i}\int_0^T|u(s)|\,ds-\frac1{\kappa_i}\int_0^Te^{-\kappa_i(T-s)}|u(s)|\,ds
\end{align*}
and
$$\int_0^T\int_t^\infty e^{\kappa_i(t-s)}|u(s)|\,ds\,dt=\frac1{\kappa_i}\int_0^\infty(e^{\kappa_i(s\wedge T-s)}-1)|u(s)|\,ds,
$$
which by~\eqref{derivatives int condition for Euler Lagrange}
 both converge to  finite limits for $T\ua\infty$. It thus follows   that $\int_0^\infty|\dot Y^*(t)|+|\ddot Y^*(t)|\,dt<\infty$. In the same way, we get  $\int_0^\infty|  Y^*(t)|\,dt<\infty$. Adding the facts that $Y^*$ is continuous and tends to zero as $t\ua\infty$, we obtain $\int_0^\infty(Y^*(t))^2\,dt<\infty$ and in turn  $Y^*\in\cX_{\text{det}}(y,\infty)\cap C^2[0,\infty)$. The optimality of $Y^*$ follows as in the second part of the proof of Lemma~\ref{one-agent optimization problem lemma}.
 \end{proof}
 

 \begin{proof}[Proof of Theorem~\ref{Nash eq exist thm infinite horizon}] One first shows just as in Lemma~\ref{Equilibrium uniqueness lemma} that there can be at most one Nash equilibrium for mean-variance optimization. Moreover, one shows as in the proof of Corollary~\ref{CARA Nash Eq exist cor} that a Nash equilibrium for mean-variance optimization is also a Nash equilibrium for CARA utility maximization.
 
 Now, we turn to the proof of existence of a Nash equilibrium for given initial values $x_1,\dots, x_n\in\bR$.  Let $M$ be the $2n\times 2n$-matrix defined in Remark~\ref{matrix rm}. As observed in the proof of Lemma~\ref{matrix M lemma}, any eigenvector of $M$ with eigenvalue $\tau$ must be of the form ${\bm v\choose\tau\bm v}$ for some $\bm v\in\bR^n$. We will show below that in both cases,~\ref{Nash eq exist thm infinite horizon (a)} and~\ref{Nash eq exist thm infinite horizon (b)},  there exists a basis $\bm v_1,\dots, \bm v_n$ of $\bR^n$ and  numbers $\tau_1,\dots,\tau_n<0$ (which are not necessarily distinct) such that ${\bm v_1\choose\tau_1\bm v_1},\dots, {\bm v_n\choose\tau_n\bm v_n}$  are eigenvectors of $M$. Taking this fact as given, let $c_1,\dots,c_n\in\bR$ be such that $c_1\bm v_1+\cdots+c_n\bm v_n=(x_1,\dots,x_n)^\top$ and define
 $$\bm Z(0):=c_1{\bm v_1\choose\tau_1\bm v_1}+\cdots+c_n {\bm v_n\choose\tau_n\bm v_n}\qquad\text{and}\qquad\bm Z(t):=e^{tM}\bm Z(0).
 $$
 We denote by $\bm X^*(t)$ the first $n$ components of $\bm Z(t)$. 
 As observed in the proof of Theorem~\ref{Nash Eq exist thm} and Remark~\ref{matrix rm}, $\bm X^*(t)$ will solve the system~\eqref{ode system 1} of coupled Euler--Lagrange equations, which by Lemma~\ref{one-agent optimization problem lemma infinite horizon} is sufficient for optimality in the infinite-horizon setting, provided that the components correspond to admissible strategies and satisfy the integrability conditions of Lemma~\ref{one-agent optimization problem lemma infinite horizon}. But each component of $\bm X^*(t)$ is by construction a linear combination of the decreasing exponential functions $e^{\tau_1t},\dots, e^{\tau_nt}$, and so these conditions are clearly satisfied.
 
 Now, we consider  case~\ref{Nash eq exist thm infinite horizon (a)}. Then $\theta_-$ and $\rho_-$ defined in~\eqref{thetarho+-} are strictly negative, and so the required existence of $\bm v_1,\dots,\bm v_n$ follows from Lemma~\ref{matrix M lemma}. It follows from the preceding part of the proof that each component of $\bm X^*(t)$ can be written as 
 $$X^*_i(t)=c_i(\theta_-)e^{\theta_-t}+c(\rho_-)e^{\rho_-t}.
 $$
 Letting again $\Sigma(t):=\sum_{j=1}^nX^*_j(t)$ and 
 arguing as in the proof of Theorem~\ref{alpha n=alpha b=0 thm} yields first that $\Sigma(t)=\sum_{i=1}^nx_ie^{\rho_-t}$ and then that
$$
c(\rho_-)=\frac1n\sum_{j=1}^nx_j\quad\text{and}\quad c_i(\theta_-)=x_i-c(\rho_-).
$$
 This establishes~\eqref{Nash eq exist thm infinite horizon strategy}
 and completes the proof of Theorem~\ref{Nash eq exist thm infinite horizon} under assumption~\ref{Nash eq exist thm infinite horizon (a)}.
 
Now we turn toward case~\ref{Nash eq exist thm infinite horizon (b)}. We may assume without loss of generality that $\sigma=1$.
The characteristic polynomial of  the matrix $M$ of the system~\eqref{ODE formula} for $n=2$ is
$$
\chi(\tau):=\tau^4-\frac{2\gamma}{3\lambda}\tau^3-\frac{\gamma^2+2\lambda(\alpha_1+\alpha_2)}{3\lambda^2}\tau^2+\frac{\alpha_1\alpha_2}{3\lambda^2}.$$
Its derivative, $\chi'$,  has three distinct roots, $t_0,t_+,t_-$, which are  given by
$$
t_0=0,\quad t_{\pm}=\frac{3 \gamma \pm \sqrt{33 \gamma ^2+48 \left(\alpha _1+\alpha _2\right)\lambda  }}{12 \lambda }.
$$
Note first that $t_0$ is a strictly positive local maximum of $\chi$ because
$$
\chi(t_0)=\frac{\alpha_1\alpha_2}{3\lambda^2}>0,\quad \chi''(t_0)=-\frac{2 \left(2 \alpha _1 \lambda  +2 \alpha _2 \lambda  +\gamma ^2\right)}{3 \lambda ^2}<0.
$$
Next, $t_+>0$, $t_-<0$, and
\begin{align*}
\chi(t_-)&=\frac1{864 \lambda ^5}
\Big(-96 \lambda ^3 \left(\alpha_1^2-\alpha_1
   \alpha_2+\alpha_2^2\right)-168 \gamma ^2 \lambda ^2
   (\alpha_1+\alpha_2)-69 \gamma
   ^4 \lambda\\
   &\qquad+16 \gamma  \lambda  (\alpha_1+\alpha_2) \sqrt{48
   \lambda ^3 (\alpha_1+\alpha_2)+33 \gamma ^2 \lambda ^2}+11 \gamma ^3
   \sqrt{48 \lambda ^3 (\alpha_1+\alpha_2)+33 \gamma ^2 \lambda ^2}\Big).
   \end{align*}
   If we can show that $\chi(t_-)<0$ then $\chi$ will have precisely two distinct strictly negative roots, due to the intermediate value theorem. 
 It is, however, not easy to determine by direct inspection of our preceding formula whether indeed $\chi(t_-)<0$. But, we already know that for $\alpha_1=\alpha_2$ the matrix $M$ has exactly two strictly negative (though not necessarily distinct) eigenvalues, $\rho_-$ and $\theta_-$. So in this case, both eigenvalues must be strictly negative roots of $\chi$. We moreover know that $\chi(0)>0$, $\lim_{\tau\da-\infty}\chi(\tau)=+\infty$, and that $t_-$ is the only strictly negative critical point of $\chi$. It follows that we must have $\chi(t_-)\le0$  when $\alpha_1=\alpha_2$. Now suppose that $\alpha_1\neq\alpha_2$ and let $\bbar\alpha:=\frac12(\alpha_1+\alpha_2)$. Then $\alpha_1+\alpha_2=\bbar\alpha+\bbar\alpha$ and
 $$\alpha_1^2-\alpha_1
   \alpha_2+\alpha_2^2-\bbar\alpha^2=\frac34(\alpha_1-\alpha_2)^2>0
 $$
It therefore follows that $\chi(t_-)<\bbar\chi(t_-)$, where $\bbar\chi$ denotes the characteristic polynomial of $M$ when both $\alpha_1$ and $\alpha_2$ have been substituted by $\bbar\alpha$. As the formula for $t_-$ is   invariant under this substitution, we must have $\bbar\chi(t_-)\le0$ according to what has been said before, and so we arrive at $\chi(t_-)<0$.

It follows from the preceding paragraph that $M$ has two distinct strictly negative eigenvalues $\tau_1$ and $\tau_2$. Hence, there exist corresponding eigenvectors of the form ${\bm v_1\choose\tau_1\bm v_1}$ and ${\bm v_2\choose\tau_2\bm v_2}$. But, we still need to exclude the possibility that $\bm v_1$ and $\bm v_2$ are linearly dependent to complete the proof. To this end, note that it follows from~\eqref{matrix M} that we must have 
\begin{align}\label{vector quadratic eqn}
\frac1\lambda\Big(A-\frac1{n+1}JA\Big)\bm w+\tau\frac\gamma\lambda\Big(I-\frac2{n+1}J\Big)\bm w=\tau^2\bm w
\end{align}
for ${\bm w\choose\tau\bm w}$ to be an eigenvector of $M$ with eigenvalue $\tau$. 

Let us first suppose that the components $w_1$ and $w_2$ of $\bm w$ do not add up to zero: $w_1+w_2\neq0$. Then, taking the inner product of the vector equation~\eqref{vector quadratic eqn}
 with the vector ${1\choose1}$ yields the equation
$$\alpha_1w_1+\alpha_2w_2-\tau\gamma(w_1+w_2)=3\tau^2\lambda(w_1+w_2).
$$
This quadratic equation in $\tau$ has the two possible roots
$$\tau_\pm=\frac{-\gamma\pm\sqrt{\gamma^2+12\lambda\frac{\alpha_1w_1+\alpha_2w_2}{w_1+w_2}}}{6\lambda},
$$
one of which must be equal to $\tau$. As $\tau_-<0<\tau_+$ it follows that ${\bm w\choose\wt\tau\bm w}$ cannot be an eigenvector of $M$ for any $\wt\tau$ that is different from $\tau$ and has the same sign as $\tau$.

Let us now consider the case in which $w_1=-w_2$. Taking the inner product of the equation~\eqref{vector quadratic eqn} with the vector ${-1\choose1}$ and using the requirement $w_1,w_2\neq0$ yields the  equation $\alpha_2+\alpha_1+2\tau\gamma =2\tau^2\lambda $, which is independent of $w_1$ and $w_2$.
It has the roots
$$\frac{\gamma\pm\sqrt{\gamma^2+4\lambda(\alpha_1+\alpha_2)}}{2\lambda},
$$
which again have different signs. We thus conclude as in the case $w_1+w_2\neq0$.
\end{proof}


\begin{proof}[Proof of Corollary~\ref{predatory cor}] It  follows from~\eqref{Nash eq exist thm infinite horizon strategy}
 that  $X_i^*(t)$ has the same sign as  
\begin{align*}
\rho_--\theta_-=\frac{\gamma}{2\lambda}\bigg(\frac{-2n}{n+1}+\sqrt{1+\xi}-\sqrt{\Big(\frac{n-1}{n+1}\Big)^2+\frac\xi{n+1}}\bigg),
\end{align*} 
 where $\xi=4\alpha\sigma^2\lambda/\gamma^2$. The right-hand side is a strictly increasing function of $\xi$ and vanishes for $\xi=8$.
\end{proof}

\bibliography{MarketImpact}{}
\bibliographystyle{agsm}

\end{document}